\documentclass[12pt]{amsart}
\usepackage{amssymb}
\usepackage{amsbsy}
\usepackage{amscd}
\usepackage{arydshln}

\newcommand{\linelabel}[1]{}

\usepackage[mathscr]{eucal}
\usepackage{nicefrac}

\usepackage[dvipdf]{graphicx}
\usepackage[all]{xy}

\usepackage{scalerel}

\makeatletter
\renewcommand{\thesubsection}{\thesection(\@roman\c@subsection)}
\makeatother
\usepackage{verbatim}
\usepackage{version}
\usepackage{color}
\definecolor{deepjunglegreen}{rgb}{0.0, 0.29, 0.29}
\definecolor{darkspringgreen}{rgb}{0.09, 0.45, 0.27}
\newenvironment{NB}{
\color{red}{\bf NB}. \footnotesize
}{}
\newenvironment{NB2}{
\color{blue}{\bf NB}. \footnotesize
}{}
\excludeversion{NB}
\excludeversion{NB2}
\newtheorem{Theorem}[equation]{Theorem}
\newtheorem{Corollary}[equation]{Corollary}
\newtheorem{Lemma}[equation]{Lemma}
\newtheorem{Proposition}[equation]{Proposition}

\theoremstyle{definition}

\newtheorem{Conjecture}[equation]{Conjecture}

\theoremstyle{remark}
\newtheorem{Remark}[equation]{Remark}

\newtheorem*{Claim}{Claim}

\numberwithin{equation}{section}

\newcommand{\thmref}[1]{Theorem~\ref{#1}}

\newcommand{\CC}{{\mathbb C}}
\newcommand{\Z}{{\mathbb Z}}

\newcommand{\RR}{{\mathbb R}}

\newcommand{\Spec}{\operatorname{Spec}\nolimits}

\newcommand{\Hom}{\operatorname{Hom}}
\newcommand{\Ext}{\operatorname{Ext}}

\newcommand{\ve}{\varepsilon}
\renewcommand{\MR}[1]{}

\renewcommand{\AA}{{\mathbb A}}

\newcommand{\Bun}{\operatorname{Bun}}

\newcommand\ZZ{\mathbb Z}
\newcommand\grb{\mathfrak b}

\newcommand\grg{\mathfrak g}

\newcommand{\Perv}{\operatorname{Perv}}

\newcommand{\TT}{\mathbb T}

\newcommand{\bigzerou}{\lower1.7ex\hbox{\Huge 0}}

\newcommand\alp{\alpha}
\newcommand\aff{\operatorname{aff}}
\newcommand\PP{\mathbb{P}}
\newcommand\calZ{\mathcal{Z}}

\newcommand\x{\times}

\newcommand{\Heis}{\mathfrak{Heis}}
\newcommand{\Vir}{\mathfrak{Vir}}
\newcommand{\IC}{\operatorname{IC}}

\newcommand{\fa}{\mathfrak a}
\newcommand{\Stab}{\operatorname{Stab}}

\newcommand{\scF}{\mathscr F}
\newcommand{\scG}{\mathscr G}

\newcommand{\cA}{\mathcal A}

\newcommand\calF{\mathcal F}
\newcommand\calA{\mathcal A}
\newcommand\tilZ{\widetilde{Z}}
\newcommand\tcalF{\widetilde{\calF}}

\newcommand{\scW}{\mathscr W}

\newcommand\grt{\mathfrak t}
\newcommand\Del{\Delta}

\newcommand{\on}{\operatorname}
\newcommand{\unl}{\underline}

\newcommand{\iso}{{\stackrel{\sim}{\longrightarrow}}}

\newcommand{\BA}{{\mathbb{A}}}
\newcommand{\BC}{{\mathbb{C}}}
\newcommand{\BG}{{\mathbb{G}}}

\newcommand{\CF}{{\mathcal{F}}}
\newcommand{\CK}{{\mathcal{K}}}

\newcommand{\CN}{{\mathcal{N}}}
\newcommand{\CO}{{\mathcal{O}}}
\newcommand{\CU}{{\mathcal{U}}}

\newcommand{\Gr}{{\operatorname{Gr}}}
\newcommand{\svee}{{\!\scriptscriptstyle\vee}}

\newcommand{\fb}{{\mathfrak{b}}}
\newcommand{\fg}{{\mathfrak{g}}}

\newcommand{\ft}{{\mathfrak{t}}}
\newcommand\calB{\mathcal B}
\newcommand\lam{\lambda}
\newcommand\bft{\mathbf t}
\newcommand\sig{\sigma}
\newcommand\del{\delta}
\newcommand\calU{\mathcal U}
\newcommand\perf{\operatorname{perf}}

\newcommand\calP{\mathcal P}
\newcommand\calH{\mathcal H}
\newcommand\Hecke{\operatorname{Hecke}}
\newcommand\calD{\mathcal D}

\newcommand{\oZ}{\mathring{Z}}
\newcommand{\oU}{\mathring{\mathcal U}}

\renewcommand\BC{\CC}
\newcommand\BP{\PP}
\newcommand\CB{\calB}
\renewcommand\on{\operatorname}
\newcommand\BN{\NN}
\newcommand\CR{\mathcal R}
\newcommand\Lam{\Lambda}
\renewcommand\BN{\mathbb N}
\newcommand\BQ{\mathbb Q}
\newcommand\BZ{\mathbb Z}
\newcommand\calC{\mathcal C}
\newcommand\Rep{\operatorname{Rep}}
\newcommand\calG{\mathcal G}
\newcommand\calO{\mathcal O}
\newcommand\calK{\mathcal K}
\newcommand\Act{\operatorname{Act}}
\newcommand\grs{\mathfrak s}
\newcommand\calM{\mathcal M}
\newcommand\grn{\mathfrak n}
\newcommand\Ho{\operatorname{Ho}}
\newcommand\Vect{\operatorname{Vect}}
\newcommand\sI{{\mathsf{I}}}
\newcommand\sF{{\mathsf{F}}}

\usepackage{url}

\makeindex

\usepackage{xr-hyper}
\usepackage[
pdftex,
bookmarks=true,
colorlinks=true,
citecolor=deepjunglegreen,
filecolor=deepjunglegreen,
debug=true,
pdfnewwindow=true]{hyperref}

\begin{document}

\title[Kazhdan-Lusztig conjecture via zastava spaces]
{Kazhdan-Lusztig conjecture via zastava spaces}

\author[A.~Braverman]{Alexander Braverman}
\address{Department of Mathematics, University of Toronto and Perimeter Institute
of Theoretical Physics, Waterloo, Ontario, Canada, N2L 2Y5;
\newline Skolkovo Institute of Science and Technology}
\email{braval@math.toronto.edu}
\author[M.~Finkelberg]{Michael Finkelberg}
\address{National Research University Higher School of Economics, Russian Federation,
  Department of Mathematics, 6 Usacheva st, 119048 Moscow;
\newline Skolkovo Institute of Science and Technology;
\newline Institute for the Information Transmission Problems}
\email{fnklberg@gmail.com}
\author[H.~Nakajima]{Hiraku Nakajima}
\address{Kavli Institute for the Physics and Mathematics of the Universe (WPI),
  The University of Tokyo, 5-1-5 Kashiwanoha, Kashiwa, Chiba, 277-8583, Japan
  \newline Research Institute for Mathematical Sciences,
Kyoto University, Kyoto 606-8502, Japan}
\email{hiraku.nakajima@ipmu.jp}
\maketitle

\begin{abstract}
  We deduce the Kazhdan-Lusztig conjecture on the multiplicities of simple modules over a
  simple complex Lie algebra in Verma modules in category $\mathcal O$ from the equivariant
  geometric Satake correspondence and the analysis of torus fixed points in zastava spaces.
  We make similar speculations for the affine Lie algebras and $\scW$-algebras.
\end{abstract}

\setcounter{tocdepth}{1}
\tableofcontents


\section{Introduction}
\label{intro}
\subsection{The setup: simple version}\label{setup}
Let $Z$ be a scheme of finite type over $\CC$ and let $\calF$ be a pure complex of constructible
sheaves on $Z$ (i.e.\ the constructible part of data of a pure complex of Hodge $D$-modules).
In particular, $\calF$ is semi-simple.

Assume that some  torus $\TT$ acts on $Z$ and that $\calF$ is $\TT$-equivariant. Moreover, we shall assume that $Z$ is conical in the following sense: there exists a cocharacter $\CC^{\times}\to \TT$ and a point $z_0\in Z$ such that the corresponding action of $\CC^{\times}$ on $Z$ extends to a morphism $\AA^1\times Z\to Z$ sending $\{ 0\}\times Z$ to $z_0$.

Let us set $\calU=\Ext^*_{\TT}(\calF,\calF)$, where the right hand side stands for $\Ext$ in the $\TT$-equivariant derived category.
This is an associative algebra over $\CC$. In~Section~\ref{abs-form} we suggest a formalism
that allows to study representation theory of $\calU$ in a geometric way. In the rest of the
paper we explain how to apply this formalism in order to reprove the Kazhdan-Lusztig conjecture
for category $\calO$ of representations of semi-simple Lie algebras (and their affine
generalizations).

Let us first give a brief outline of the contents of~Section~\ref{abs-form} (and later we shall give a sketch explaining how to apply this to study representations of simple Lie algebras).

\subsection{Specialization of the central character and hyperbolic restriction}\label{hyp-int}
We have the obvious homomorphism $H^*_{\TT}(\on{pt},\CC)\to \calU$ whose image lies
in the center. The assumption that $\calF$ is pure together with the assumption that $Z$ is
conical implies that this map is injective. For simplicity of the discussion let us assume that $\calU$ is free over $H^*_{\TT}(\on{pt},\CC)$ (we are going to relax this assumption a little later).  For any $x\in \grt=\Spec(H^*_{\TT}(\on{pt},\CC))$ let us
denote by $\calU_x$ the corresponding specialization of $\calU$. This is a finite-dimensional associative algebra. Any irreducible representation of $\calU$ factors through
$\calU_x$ for some $x$. Thus in order to proceed we need to realize $\calU_x$ geometrically.

Let $\TT_x$ be the subtorus of $\TT$ generated by $x$. Then we can consider the functor of hyperbolic restriction
$\Phi_x\colon  D_\TT^b(Z)\to D_\TT^b(Z^{\TT_x})$ (here $D_\TT^b(Z)$ stands for the derived category
of $\TT$-equivariant constructible sheaves on $Z$) --- we recall the definition of $\Phi_x$
in~\S\ref{hyperbolic}. Let $\calF_x=\Phi_x(\calF)$.  We explain in~Section~\ref{abs-form} that under the above assumptions
we get a natural morphism $\calU_x\to \Ext(\calF_x,\calF_x)$ (note that in the RHS we deal with usual (i.e.\ non-equivariant) Ext's).
We do not have any general way to guarantee that this map is an isomorphism but often it is (one such instance is recalled in~\S\ref{spec-easy}). However, if we assume that this map is an isomorphism, we can do the following:

1) Classify geometrically irreducible representations of $\calU_x$;

2) Construct geometrically certain modules over $\calU_x$ that we call standard and costandard modules.

3) Compute geometrically the multiplicities of simple modules in standard (or costandard) ones.

\noindent
Namely, the complex $\calF_x$ is semi-simple. Let us assume that there exists a stratification $Z^{\TT_x}=\bigsqcup\limits_{w\in W_x} Z^{\TT_x}_w$ such that
$\calF_x$ is constructible with respect to the above stratification (here $W_x$ is some finite index set). Thus $\calF_x$ is isomorphic to a direct sum of the
form
\[
\bigoplus\limits_{w\in W_x, \sigma\in\operatorname{Irr}(\pi_1(Z^{\TT_x}_w))} \calF_{x,w,\sigma}\otimes L_{x,w,\sigma}
\]
where $\calF_{x,w,\sigma}$ is the simple perverse sheaf on $Z^{\TT_x}$ which is equal to the
Goresky-MacPherson extension
of the irreducible local system on $Z^{\TT_x}_w$ given by $\sigma$, and $L_{x,w,\sigma}$ is a (graded) vector space.
It is easy to see that every $L_{x,w,\sigma}$ acquires a natural action of $\calU_x$ such that

a) This action is irreducible.

b) The assignment $(w,\sigma)\mapsto L_{x,w,\sigma}$ is a bijection between pairs $(w,\sigma)$ such that $L_{x,w,\sigma}\neq 0$, and
isomorphism classes of simple modules over $\calU_x$.

\noindent
Moreover, for every $w\in W_x$ let $i_w$ be the embedding of some point of $Z^{\TT_x}_w$ into $Z^{\TT_x}$.  Define the standard module
$\Del_{x,w,\sig}$ as the $\sigma$-part of $i_w^!\calF_x$ (similarly, we define the costandard module $\nabla_{x,w,\sigma}$ as the $\sigma$-part of $i_w^*\calF_x$).

Assume for example, that all $L_{x,w,\sigma}$ and $\Del_{x,w,\sigma}$ are equal to $0$ unless $\sigma$ is trivial (in this case we shall omit $\sigma$ from the notation).
Then it follows from the above that the multiplicity of $L_{x,w}$ in $\Del_{x,w'}$ is equal to the total dimension of $i_{w'}^!\calF_{x,w}$.

\begin{NB}
  The following paragraph is added by HN on July 15.
\end{NB}%

This framework is similar to one used for Kazhdan-Lusztig type
character formula in other contexts, e.g.,
\cite{CG,GinzburgVasserot,Lu-cus2,Vasserot,Na-qaff,handsaw}.
In these cases, $\calF=\pi_*\CC_{\tilZ}$, where $\pi\colon \tilZ\to Z$
is a proper $T$-equivariant morphism with $\tilZ$ smooth, or $\calF$
with some vanishing cohomology assumptions (\cite{Lu-cus2}).
Our new point is to use the hyperbolic restriction functor $\Phi_x$ in
order to drop the existence of $\pi\colon\tilZ\to Z$ or the vanishing
assumptions. See Section~\ref{abs-form} for more detail.

\subsection{Zastava spaces}
We would like to apply the above analysis to the case when $\calU$ is some close relative of the universal enveloping algebra of a simple Lie algebra over $\CC$. Of course, the latter is of infinite rank over its center, so the above formalism cannot be applied literally. However, we can extend the formalism in the following way: we can replace $Z$ by an ind-scheme $Z=\underset{\to}\lim\,  Z_i$ where each $Z_i$ is
finite-dimensional and the limit is taken with respect to closed $\TT$-equivariant embeddings. Also we assume that the sheaf $\calF$ is of the form $\oplus \calF_i$ where each $\calF_i$ is supported on $Z_i$. In this case for a finitely generated algebra $\calU$ we can only expect an injective map $\calU\to \Ext(\calF,\calF)$ (which can become an isomorphism after certain completion).
The definitions of $L_{x,w}$ and $\Del_{x,w}$ still make sense, and if we assume that every
$L_{x,w}$ is simple (together with some additional assumptions, cf.~\S\ref{ind-int} for the
details), in this case we get a similar formula for the multiplicity of
$L_{x,w}$ in every $\Del_{x,w'}$.

Let us discuss which $Z$ we want to work with. Let $G$ be an almost simple simply connected
algebraic group over $\CC$ with Lie algebra $\grg$.
Let $\Lam$ be its coroot lattice; we denote by $\Lam_+$ the sub-semigroup consisting of sums
of positive coroots. For every $\alp\in \Lam_+$ we can consider the scheme $Z^{\alp}$ of
{\em based quasi-maps} from $\PP^1$ to the flag variety $\calB$ of $G$ of degree $\alp$
(we recall the definition in~Section~\ref{fixed points}). This is an affine scheme over $\CC$ acted on by the torus $\TT=T\times \CC^{\x}$, where $T$ is a maximal torus of $G$ (here $T$ acts on $\calB$ and $\CC^{\x}$ acts on $\PP^1$).  For every $\alp\geq \beta$ we have a natural closed $\TT$-equivariant embedding $Z^{\beta}\hookrightarrow Z^{\alp}$ (given by ``adding defect at $0$") and we denote by $Z$ the corresponding ind-scheme; we also let $\calF=\oplus_{\alp\in \Lam_+} \calF^{\alp}$ where $\calF^{\alp}$ is the IC-sheaf of $Z^{\alp}$.

\subsection{The universal enveloping algebra and its relatives}
\label{relatives}
Let us now discuss which algebra $\calU$ we want to work with. Let $\grg^\svee$ denote the
Langlands dual Lie algebra.
Let $U(\grg^\svee)$ be its universal enveloping algebra and let $U_{\hbar}(\grg^\svee)$ be its ``renormalized" version.
By definition $U_\hbar(\grg^\svee)$ is a graded $\CC[\hbar]$-algebra generated by the linear subspace $\grg^\svee$ with relations
$XY-YX=\hbar[X,Y]$ for $X,Y\in \grg^\svee$; the grading is defined in such a way that $\deg(X)=\deg(\hbar)=2$ for $X\in \grg^\svee$.

Let $\calZ_{\hbar}(\grg^\svee)$ denote the center of $U_{\hbar}(\grg^\svee)$. It follows from the
Harish-Chandra homomorphism that the center $\calZ_{\hbar}(\grg^\svee)$ is naturally isomorphic to
$S_{\hbar}(\bft^*)^W\simeq S(\grt^*)^W$
(here as before $\grt={\mathbf t}\oplus\CC$ denotes the Lie algebra of $\TT$).\footnote{In
  what follows we shall always treat $S(\grt^*)$ as a graded algebra with generators having
  degree 2; in this way we can identity $S_{\hbar}(\bft^*)$ with $S(\grt^*)$ as graded algebras.}
We now set
\[
\calU=U_{\hbar}(\grg^\svee)\underset{\calZ_{\hbar}(\grg^\svee)}\otimes S(\grt^*).
\]
Note that the center of $\calU$ is $S(\grt^*)$ and for any $x\in \grt$ of the form $(\mu,n)$
with $\mu\in \bft$ and $n\in \CC,\ n\neq 0$, the specialization $\calU_x$ of $\calU$ at $x$ is
naturally isomorphic to the specialization of $U(\grg^\svee)$ at $\lam:=n^{-1}\mu$. Hence any
$\calU_x$-module can be regarded as a $\grg^\svee$-module with central character $\lam$.

\subsection{What is actually done in this paper}
The main results of this paper can be summarized as follows:

\begin{enumerate}
\item In Section \ref{abs-form} we develop carefully the abstract formalism discussed
  in~\S\ref{setup} and~\S\ref{hyp-int}.

\item In~Section~\ref{action} we construct an (injective) homomorphism
  $\calU\to \Ext(\calF,\calF)$, satisfying some natural properties (cf.~Theorem~\ref{main1}).
  The construction is an easy corollary of the derived Satake equivalence (that we also recall
  in~Section~\ref{action}).

\item In Section~\ref{fixed} we compute the fixed point schemes $Z^{\TT_x}$ with the corresponding
  stratifications. Here $x$ is an element of $\Spec(H^*_{\TT}(pt))$, i.e.\ it is a pair
  $(\mu,n)$ where $\mu\in \bft$ and $n\in \CC$. For simplicity, we restrict ourselves to the
  case when $n\in\ZZ_{>0}$, $\mu$ is integral and the ratio $\lam=\mu/n$ is dominant (i.e.\ it
  is a dominant rational weight).
In this case the scheme $Z^{\TT_x}$ is always finite-dimensional, and in fact it can be identified with the open Bruhat cell in some partial flag variety of a certain (endoscopic) reductive subgroup of $G$ (the choice of this flag variety is dictated by integrality properties of $\lambda$); moreover, the corresponding stratification comes from certain (opposite) parabolic Schubert stratification of the above flag variety (the type of the parabolic is determined by regularity properties of $\lambda$).

\item In Section \ref{standard} we compute the corresponding standard and co-standard modules $\Del_{x,w}$ and $\nabla_{x,w}$ and identify them with Verma and dual Verma modules for the Lie algebra $\grg^{\vee}$ with appropriate highest weights. As a byproduct it follows that the modules $L_{x,w}$ are the corresponding simple modules in category $\calO$ for $\grg^{\vee}$. Together with the fixed points analysis discussed in~(3) above this implies the Kazhdan-Lusztig conjecture (formally, we do it only for rational central characters, but the generalization to arbitrary central character is straightforward --- we don't perform it in this paper in order not to overload the reader with cumbersome notation). In fact, we prove the Kazhdan-Lusztig conjecture in the formulation that
  is ``Koszul dual" to the usual one (cf.~\S\ref{kl}), but the two formulations are known to be equivalent. Note also that the formulation of the Kazhdan-Lusztig conjecture usually has to be adjusted to the integrality and regularity properties of the central character --- here we get it algorithmically from the fixed points analysis discussed above.

\item In Sections \ref{affine case} and \ref{twi} we discuss a generalization of the above to
  the case of affine Lie algebras (and also the corresponding $\scW$-algebras). In this
  case, one can define the affine zastava spaces and perform similar fixed points analysis. Unfortunately, the derived Satake equivalence is not available in the affine case, so our construction of the homomorphism $\calU\to \Ext(\calF,\calF)$ does not work in the affine case. However, we believe that this homomorphism can also be defined directly (by explicitly describing the images of the Chevalley generators)
  and such a definition should go through in the affine case; we postpone this for another publication. Assuming that the above homomorphism can be constructed, we can again reprove the Kazhdan-Lusztig conjecture in the affine case. One nice feature of this approach is that it treats the cases of representations of affine Lie algebras of positive, negative and critical level uniformly (one just has to compute the fixed point schemes $Z^{\TT_x}$ as above with the natural stratification). Similarly, in this way we can potentially
  \begin{NB}
    added by HN on July 15.
  \end{NB}%
reprove the version of the Kazhdan-Lusztig conjecture for
  $\scW$-algebras (proved earlier by Arakawa).

\end{enumerate}

\subsection{Acknowledgments}
Part of this work was done while the authors visited the Simons Center for Geometry and Physics
in October 2013. We are grateful to SCGP for the wonderful working conditions.
\begin{NB}
  Added by HN on July 15.
\end{NB}%
Results were reported at a few occasions around 2013-14. We apologize
for the long delay in writing up the paper.
A.B.\ was partially supported by NSERC.
M.F.\ was partially funded within the framework of the HSE University
Basic Research Program and the Russian Academic Excellence Project ‘5-100’.
H.N.\ was partially supported by the World Premier International Research Center Initiative
(WPI Initiative), MEXT, Japan, and JSPS Kakenhi Grant Numbers
23224002, 
23340005, 
24224001, 
25220701, 
16H06335, 
19K21828. 
\section{Abstract formalism}\label{abs-form}
In this Section we keep the notations of~\S\ref{setup}. We first would like to discuss carefully how to realize the algebra $\calU_x$ geometrically (for some $x\in \grt$).

\subsection{Specialization of the central character: the easy case}\label{spec-easy}
Let us first discuss a familiar example.
Assume that we are given a proper $\TT$-equivariant morphism $\pi\colon \tilZ\to Z$ with $\tilZ$ smooth. Moreover, let us assume that $\calF=\pi_*\CC_{\tilZ}$.
In this case $\calU_x$ admits the following description.

Let $\TT_x$ be the subtorus in $\TT$ defined as the Zariski closure of $\exp(tx)$ for all $t\in \CC$.
Then we can consider the fixed point sets $\tilZ^{\TT_x}$ and $Z^{\TT_x}$. There is a natural
proper morphism $\pi_x\colon \tilZ^{\TT_x}\to Z^{\TT_x}$.
Let $\tcalF_x=(\pi_x)_*\CC$. The following result is contained in~\cite[Chapter 8]{CG}:

\begin{Lemma}\label{kl-ginz} Assume that $\tilZ\underset{Z}\times \tilZ$ has an algebraic cell decomposition. Then $\calU$ is free over $H_{\TT}^*(\on{pt},\CC)$ and $\calU_x=\Ext^*(\tcalF_x,\tcalF_x)$.
\end{Lemma}

Let us continue to work under the assumptions of Lemma \ref{kl-ginz}.

The complex $\tcalF_x$ is semi-simple. Let us assume that there exists a stratification $Z^{\TT_x}=\bigsqcup\limits_{w\in W_x} Z^{\TT_x}_w$ such that
$\tcalF_x$ is constructible with respect to the above stratification (here $W_x$ is some finite index set). Thus $\tcalF_x$ is isomorphic to the direct sum of the
form
\[
\bigoplus\limits_{w\in W_x, \sigma\in\operatorname{Irr}(\pi_1(Z^{\TT_x}_w))} \calF_{x,w,\sigma}\otimes L_{x,w,\sigma}
\]
where $\calF_{x,w,\sigma}$ is the simple perverse sheaf on $Z^{\TT_x}$ which is equal to the
Goresky-MacPherson extension
of the irreducible local system on $Z^{\TT_x}_w$ given by $\sigma$, and $L_{x,w,\sigma}$ is a graded vector space.

It now follows easily that each $L_{x,w,\sigma}$ (with grading disregarded) is an irreducible
representation of $\Ext^*(\tcalF_x,\tcalF_x)$ and each irreducible representation of
$\Ext^*(\tcalF_x,\tcalF_x)$ appears in this way exactly once. See~\cite[Chapter 8]{CG}.
\begin{NB}
  Added by HN on July 15.
\end{NB}%
So from~Lemma~\ref{kl-ginz} we get a complete classification of irreducible $\calU_x$-modules.
\subsection{Digression on hyperbolic restriction}\label{hyperbolic}
We would like to extend the above construction to more general $\calF$.
For this we first need to recall some well-known facts about the notion of hyperbolic restriction.

Let $Z$ and $\TT$ be as above.
Let $A\subset\TT$ be a subtorus and $Z^A$ denote the
fixed point set.

Let $X_*(A)$ be the lattice of cocharacters of $A$. It is a free
$\Z$-module. Let
\begin{equation}
  \fa_\RR = X_*(A)\otimes_\Z\RR.
\end{equation}

Let $\Stab_z\subset A$ be the stabilizer subgroup of a point $z\in Z$. A {\it
  chamber\/} $\mathfrak C$ is a connected component of
\begin{equation}
    \label{eq:29}
    \fa_\RR \setminus \bigcup_{z\in Z\setminus Z^A} X_*(\Stab_z)\otimes_\Z\RR.
\end{equation}

We fix a chamber $\mathfrak C$. Choose a cocharacter $\lambda$ in
$\mathfrak C$. Let $z\in Z^A$. We introduce {\it attracting\/} and
{\it repelling\/} sets:
\begin{multline}
    \label{eq:3}
    \mathcal A_z = \left\{y \in Z \,\middle|\,
      \begin{minipage}{.6\linewidth}
      the map $t\mapsto \lambda(t)(y)$
      extends to a map $\AA^1 \to Z$ sending $0$ to $z$
      \end{minipage}
    \right\},
    \\
    \mathcal R_z = \left\{y \in Z\,\middle|\,
      \begin{minipage}{.6\linewidth}
      the map $t\mapsto \lambda(t^{-1})(y)$ extends
      to a map $\AA^1 \to Z$ sending $0$ to $z$
      \end{minipage}\right\}.
\end{multline}
These are locally closed subvarieties of $Z$, and independent of the choice
of $\lambda\in\mathfrak C$.
 Similarly we can define $\mathcal
A_Z$\index{AX@$\mathcal A_X$ (attracting set)}, $\mathcal
R_Z$\index{RX@$\mathcal R_X$ (repelling set)} if we do not fix the
point $z$ as above. These are closed subvarieties of $Z$ if we assume that $Z$ is affine; from now on we shall assume that this is the case. Note that $Z^{A}$ is a closed subvariety of both
$\mathcal A_Z$ and $\mathcal R_Z$; in addition we have the natural
morphisms $\mathcal A_Z \to Z^A$ and $\mathcal R_Z \to Z^A$.

Let us now choose a
chamber in $\fa_\RR$, and consider the diagram
\begin{equation}\label{eq:113}
  Z^A\overset{p}{\underset{i}{\leftrightarrows}}
  \cA_Z \xrightarrow{j} Z,
\end{equation}
where $i$, $j$ are embeddings, and $p$ is defined by $p(y) =
\lim_{t\to 0}\lambda(t)y$.

We consider Braden's {\it hyperbolic restriction functor\/}
\cite{Braden} (see also \cite{DrGa}) defined by $\Phi = i^*j^!$.\index{UZphi@$\Phi$ (hyperbolic restriction functor)}
\begin{Theorem}\cite{Braden,DrGa}
We have a canonical isomorphism
\begin{equation}\label{eq:Braden}
  i^* j^!
  \cong  i_-^! j_-^*
\end{equation}
on weakly $A$-equivariant objects, where $i_-$, $j_-$ are defined
as in \eqref{eq:113} for $\mathcal R_Z$ instead of $\mathcal A_Z$.
\end{Theorem}
Braden proved this theorem for a normal algebraic variety. In general this result is proved
in \cite[Theorem~3.1.6]{DrGa}.

Note also that $i^*$ and $p_*$ are isomorphic on weakly equivariant
objects, we have $\Phi = p_*  j^!$. (See \cite[(1)]{Braden}.)

\subsection{Localization theorem} \label{localization}
(See \cite[Ch.~3]{2014arXiv1406.2381B} and references therein.)
Let $\scF$, $\scG\in D^b_{\TT}(Z)$ and let $\varphi\colon Y\hookrightarrow Z$ be the embedding of a closed $\TT$-invariant subset of $Z$ which contains $Z^{\TT}$. Then
\begin{gather}\label{eq:31}
  \Ext_{{\TT}}(\scF,\scG)\to \Ext_{{\TT}}(\varphi^! \scF, \varphi^! \scG),
  \\
  \Ext_{{\TT}}(\scF,\scG)\to \Ext_{{\TT}}(\varphi^* \scF, \varphi^* \scG)\label{eq:38}
\end{gather}
are isomorphisms after inverting an appropriate element $f$ of $S(\grt^*)$. Moreover, the zeros
of $f$ are located on the corresponding walls. More precisely, if $f(x)=0$ then
$Z^{\TT}\neq Z^{\TT_x}$. In particular,
if $Z^{\TT}=Z^{\TT_x}$ then the above maps are isomorphisms in a neighbourhood of $x$.
In particular, this is true for any $x$ such that $\TT=\TT_x$.

Let $\scF\in D^b_\TT(Z)$. The homomorphism
\begin{equation}\label{eq:48}
  H^*_\TT(Z^A, i^* j^! \scF) \cong H^*_\TT(Z^A,p_* j^!\scF)
  = H^*_\TT(\cA_Z, j^! \scF) \to
 H^*_\TT(Z, \scF)
\end{equation}
becomes an isomorphism after inverting a certain element by the above
localization theorem, applied to the pair
$\cA_Z\subset Z$.

We also have two naive restrictions
\begin{equation}\label{eq:naive}
  H^*_\TT(Z^A, (j\circ i)^! \scF), \quad
  H^*_\TT(Z^A, (j\circ i)^* \scF).
\end{equation}
For the first one, we have a homomorphism to the hyperbolic restriction
\begin{equation}\label{eq:26}
  H^*_\TT(Z^A, (j\circ i)^! \scF)\to H^*_\TT(Z^A, i^* j^! \scF),
\end{equation}
which factors through $H^*_\TT(\cA_Z, j^! \scF)$. Then it also becomes
an isomorphism after inverting an element.

The second one in \eqref{eq:naive} fits into a commutative diagram
\begin{equation}
  \begin{CD}
    H^*_\TT(Z^A, i^* j^! \scF)
    @>>>
    H^*_\TT(Z^A, (j\circ i)^* \scF)
\\
@AAA @AAA
\\
    H^*_\TT(\cA_Z, j^! \scF) @>>>
    H^*_\TT(\cA_Z, j^* \scF).
  \end{CD}
\end{equation}
Two vertical arrows are isomorphisms after inverting an element
$f$. The lower horizontal homomorphism factors through $H^*_\TT(Z,\scF)$
and the resulting two homomorphisms are isomorphisms after
inverting an element, which we may assume equal to $f$. Therefore the
upper arrow is also an isomorphism after inverting $f$.

\subsection{Specialization of the central character: the general case}
\label{general case}
Let us apply the above construction to the case when $A=\TT_x$.
We denote by $\grt_x$ the Lie algebra of $\TT_x$. It is clear that $x\in \fa$ does not lie on
any of the walls. Thus we have a canonical choice of chamber: the one which contains $x$.  We let
$\Phi_x\colon  D_\TT^b(Z)\to D_\TT^b(Z^{\TT_x})$ denote the corresponding hyperbolic restriction functor.

We now set $\calF_x=\Phi_x(\calF)$.
Note that even in the case when $\calF=\pi_*\CC_{\tilZ}$ for some $\tilZ$ we do not claim that
$\calF_x$ is the same as $\tcalF_x$ (although there are many interesting cases when it is
indeed the case, cf.~\cite{tensor2}).

\begin{Lemma}
  There is a natural homomorphism $\calU_x\to \Ext(\calF_x,\calF_x)$.
\end{Lemma}

\begin{proof}
We view $\calF_x$ as an object of the equivariant derived category and we have a natural homomorphism
$\calU\underset{S(\grt^*)}\otimes S(\grt^*_x)\to \Ext_{\TT_x}(\calF_x,\calF_x)$. On the other hand,
we have (a non-canonical) isomorphism
$\Ext_{\TT_x}(\calF_x,\calF_x)\simeq\Ext(\calF_x,\calF_x)\otimes S(\grt_x^*)$ since
$\calF_x$ is pure (due to our standing assumption of purity of $\calF$ and the fact that
the hyperbolic restriction preserves purity).
Indeed, by definition, $\Ext_{\TT_x}(\calF_x,\calF_x)=\Ext_{BZ^{\TT_x}}(B\calF_x,B\calF_x)$, where
  $BZ^{\TT_x}$ is the classifying space $Z^{\TT_x}\stackrel{\TT_x}{\times}E\TT_x=Z^{\TT_x}\times B\TT_x$,
  and $B\calF_x\simeq\calF_x\boxtimes\CC_{B\TT_x}$ due to purity of $\calF_x$ (this follows from the fact, that the purity of $\calF_x$ implies that it is a semi-simple complex, i.e.\ as an object of the derived category of $\TT_x$-equivariant sheaves on $Z^{\TT_x}$ it is isomorphic to the direct sum of the form $\oplus_i \calG_i[i]$ where every $\calG_i$ is a $\TT_x$-equivariant perverse sheaf $Z^{\TT_x}$; on the other hand, it is known that for perverse sheaves equivariance with respect to a connected group is a property rather than a structure -- i.e.\ it is (canonically) unique if it exists. In particular, since $\TT_x$ acts trivially on $Z^{\TT_x}$, it follows that for every $\TT_x$-equivariant perverse sheaf $\calG$ on $\ZZ^{\TT_x}$ we have a canonical isomorphism $B\calG=\calG\boxtimes \CC_{B\TT_x}$. Note that the isomorphism $B\calF_x\simeq\calF_x\boxtimes\CC_{B\TT_x}$ might still be not canonical, since the decomposition of $\calF_x$ into a direct sum of shifted perverse sheaves is not canonical).
  Thus, $\Ext_{BZ^{\TT_x}}(B\calF_x,B\calF_x)\simeq\Ext_{Z^{\TT_x}}(\calF_x,\calF_x)\otimes H^*(B\TT_x,\BC)
    =\Ext_{Z^{\TT_x}}(\calF_x,\calF_x)\otimes S(\grt_x^*)$.

Hence we get a homomorphism
$\calU\underset{S(\grt^*)}\otimes S(\grt^*_x)\to \Ext(\calF_x,\calF_x)\otimes S(\grt_x^*)$. Since $x\in \grt_x$ we can specialize the above homomorphism and we get a homomorphism  $\calU_x\to \Ext(\calF_x,\calF_x)$.
\end{proof}

We shall say that $x$ is {\em very good} if the map $\calU_x\to \Ext(\calF_x,\calF_x)$
is an isomorphism.

Unfortunately, we do not have any general criterion to decide when an element $x$ is very good. Let us assume, however,
that it is the case. Then the analysis of~\S\ref{spec-easy} goes through (with $\tcalF_x$ replaced by $\calF_x$). In particular (changing slightly the notation), we can now write
\begin{equation}\label{decomp}
\calF_x=\bigoplus\limits_{w\in W_x, \sigma\in\operatorname{Irr}(\pi_1(\ZZ^{\TT_x}_w))} \calF_{x,w,\sigma}\otimes L_{x,w,\sigma}.
\end{equation}
Then
the irreducible representations of $\calU_x$
are in one-to-one correspondence with pairs $(w,\sigma)$ for which $L_{x,w,\sigma}\neq 0$  and the corresponding irreducible representations are realized in $L_{x,w,\sigma}$.

\linelabel{good}
We say that $x$ is \emph{good} (but not necessarily very good) if every $L_{x,w,\sigma}$ is an
irreducible module over $\calU_x$ and $L_{x,w,\sigma}$ is isomorphic to $L_{x,w',\sigma'}$ if and only if
$w=w',\sigma=\sigma'$.
\subsection{Standard and co-standard modules}
For any (not necessarily good or very good) $x$ and for any $(w,\sigma)$ as in~\S\ref{general case}
we can define two modules $\Del_{x,w,\sigma}$ and $\nabla_{x,w,\sigma}$ over $\calU_x$ in the following way. Let $i_{x,w}\colon  \on{pt}\to Z^{\TT_x}_w$ be the embedding of (any) point into $Z^{\TT_x}_w$. Then we set
\begin{equation}
\Del_{x,w,\sigma}=\Hom_{\pi_1(Z^{\TT_x}_w)}(\sigma, i_{x,w}^!\calF_x),\qquad \nabla_{x,w,\sigma}=\Hom_{\pi_1(Z^{\TT_x}_w)}(\sigma, i_{x,w}^*\calF_x)
\end{equation}
(here we regard $i_{x,w}^!\calF_x$ and $i_{x,w}^*\calF_x$ just as vector spaces
(or rather modules over the fundamental group of the corresponding stratum),
i.e.\ we ignore the grading).
We shall call $\Del_{x,w,\sigma}$ \emph{the standard module} corresponding to $(x,w,\sigma)$. Similarly, we shall call $\nabla_{x,w,\sigma}$
\emph{the costandard module} corresponding to $(x,w,\sigma)$.

The canonical morphism of functors $i_{x,w}^!\to i_{x,w}^*$ gives rise to a homomorphism $\eta_{x,w,\sigma}\colon  \Del_{x,w,\sigma}\to\nabla_{x,w,\sigma}$
of $\calU_x$-modules. The image of $\eta_{x,w,\sigma}$ is naturally isomorphic to $L_{x,w,\sigma}$.

Let us now define a partial order on $W_x$ by setting $w_1\leq  w_2$ if $Z^{\TT_x}_{w_1}$ lies in the closure of $Z^{\TT_x}_{w_2}$.
Also, for simplicity, let us assume that $\Del_{x,w,\sigma}=\nabla_{x,w,\sigma}=0$ unless $\sigma$ has finite image (later on we shall even assume that only trivial $\sigma$ appears).
Then we claim that $\Del_{x,w,\sigma}$ has a filtration with terms corresponding to pairs $(w',\sigma')$ with $w'\geq w$ and with successive quotients being isomorphic to direct sums of $\dim \Hom_{\pi_1(Z^{\TT_x}_w)}(\sigma, i_{x,w}^!\calF_{x,w',\sigma'})$ copies of
$L_{x,w',\sigma'}$ (similarly for $\nabla_{x,w,\sigma}$ with $i_{x,w}^!$ replaced by $i_{x,w}^*$). Let us stress that here what we mean by the number $\dim \Hom_{\pi_1(Z^{\TT_x}_w)}(\sigma, i_{x,w}^!\calF_{x,w',\sigma'})$ is just the sum of dimensions of all cohomology spaces of the complex $\Hom_{\pi_1(Z^{\TT_x}_w)}(\sigma, i_{x,w}^!\calF_{x,w',\sigma'})$ (and not its Euler characteristics).

In particular, assume that $x$ is good and that $\Del_{x,w,\sigma}=0$ unless $\sigma$ is trivial (in particular, it implies that only trivial $\sigma$'s appear in (\ref{decomp})). In this case we shall omit $\sigma$ from the notation.
Then we see that the multiplicity of the irreducible module $L_{x,w'}$ in a standard module $\Del_{x,w}$ is equal to $\dim i_{x,w}^!\calF_{x,w'}$ and any simple subquotient of $\Del_{x,w}$ is isomorphic to some $L_{x,w'}$. Moreover, in this case $\calF_{x,w'}$ is the IC-sheaf of the closure of $Z^{\TT_x}_{w'}$.

\subsection{The ind-scheme version}\label{ind-int} In what follows we need a slight generalization of the above formalism. Namely, let us assume that $Z$ is an ind-scheme of ind-finite type over $\CC$. That means that $Z$ is the direct limits of schemes $Z^{\alpha}$ of finite type over $\CC$. We allow $\calF$ to be any pure ind-finite complex. In other words, $\calF$ is the direct sum of complexes $\calF^{\alpha}$  where all $\calF^{\alpha}$ are pure (of the same weight) and the support of $\calF^{\alpha}$ lies in $Z^{\alpha}$.

In this case,  the algebra $\Ext^*_{\TT}(\calF,\calF)$ might be too big, so we shall just assume that we are given a subalgebra $\calU$ of $\Ext^*_{\TT}(\calF,\calF)$ containing $S(\grt^*)$ which is free over it. Then we can still define $\calF_x$ as $\Phi_x(\calF)$ and we still get a homomorphism
$\calU_x\to \Ext^*(\calF_x,\calF_x)$ and the analysis of the previous subsections goes through (although the modules $\Del_{x,w,\sigma},\nabla_{x,w,\sigma}$ and $L_{x,w,\sigma}$ might become infinite-dimensional).
Recall the definition of a \emph{good} rational element $x\in \grt^*$ from the end
of~\S\ref{general case}. In the present ind-scheme case we additionally require that


a) $Z^{\TT_x}$ is a scheme of finite type over $\CC$ (i.e.\ the limit
$\underset{\rightarrow}\lim (Z^{\alpha})^{\TT_x}$ stabilizes).

\noindent
In particular, if $x$ is good, then the set $W_x$ is finite. Assuming again that all $\Del_{x,w,\sigma}$ are 0 when $\sigma$ is non-trivial we see that in this case we have the following:

\textup{(i)} Each $\Del_{x,w}$ has finite length

\textup{(ii)} The only simple constituents of $\Del_{x,w}$ are of the form $L_{x,w'}$ for some $w'\geq w$

\textup{(iii)}
The multiplicity of $L_{x,w'}$ in $\Del_{x,w}$ is equal to
$\dim i_{x,w}^! \operatorname{IC}(\overline{Z^{\TT_x}_{w'}})$.

\subsection{Another definition of standard and costandard modules}
Let $x,w$ be as above and let $z\in Z_w^{\TT_x}$; we shall denote by $s_{x,w}$ the embedding of the
point $z$ into $Z$.\footnote{Recall that we denote by $i_{x,w}$ the embedding of $z$ into
  $\ZZ^{\TT_x}_w$.} To simplify the discussion we shall again assume that
$\Delta_{x,w,\sigma}=\nabla_{x,w,\sigma}=0$ if $\sigma$ is non-trivial.
Consider now $s_{x,w}^!\calF$ and $s_{x,w}^*\calF$ as objects of
$D_{\TT_x}(\on{pt})$. The algebra $\calU\underset{S(\grt^*)}\otimes S(\grt^*_x) $ then naturally maps to $\Ext^*_{\TT_x}(s_{x,w}^!\calF,s_{x,w}^!\calF)$ and to $\Ext^*_{\TT_x}(s_{x,w}^*\calF,s_{x,w}^*\calF)$. Hence $\calU\underset{S(\grt^*)}\otimes S(\grt^*_x) $  acts on
$H^*_{\TT_x}(s_{x,w}^!\calF)$ and $H^*_{\TT_x}(s_{x,w}^*\calF)$.

\begin{Lemma}\label{standard-other}
We have natural isomorphisms of $\calU_x$-modules
\[
\Delta_{x,w}\simeq (H^*_{\TT_x}(s_{x,w}^!\calF))_x;\quad \nabla_{x,w}\simeq (H^*_{\TT_x}(s_{x,w}^*\calF))_x.
\]
Here the subscript $x$ stands for the specialization at $x$ with respect to the action of $S(\grt_x^*)$.
\end{Lemma}

\begin{proof}
We prove the lemma for $\Delta_{x,w}$ (the proof for $\nabla_{x,w}$ is completely analogous).
By the definition we have
$\Delta_{x,w}=(H^*_{\TT_x}(i_{x,w}^!\calF_x))_x$. Let $\calA_x\subset Z$ denote the
corresponding attracting set and let $j_x\colon \calA_x\to Z,\ i_x\colon Z^{\TT_x}\to \calA_x$
be the corresponding embeddings. Since  $\calF_x=i_x^*j_x^!\calF$ we get a natural map
$(i_x\circ j_x)^!\calF\to \calF_x$. This is a map of objects of the equivariant derived
category $D_{\TT_x}(Z^{\TT_x})$. Hence we get a canonical map
$s_{x,w}^!\calF\to i_{x,w}^!\calF_x$ of objects of $D_{\TT_x}(\{z\})$. Applying the functor $H^*_{\TT_x}$ we get a morphism
\begin{equation}\label{morphism}
H^*_{\TT_x}(s_{x,w}^!\calF)\to H^*_{\TT_x}(i_{x,w}^!\calF_x).
\end{equation}
Since the construction of the map (\ref{morphism}) was completely canonical, it commutes with the action of $\calU$ on both sides. The algebra $S(\grt_x^*)$ acts on both sides and it is enough now to check that this map becomes an isomorphism after tensoring both sides with the field of fractions of $S(\grt_x^*)$, i.e.\ it is an isomorphism at the generic point of $\grt_x^*$. Moreover, it is an isomorphism in a neighbourhood of $x\in \grt_x$. This follows from~\S\ref{localization}
and~\cite[(6.2)]{GKM}: the
  localization isomorphism holds true away from the union of Lie algebras of stabilizers
  in $\TT$ of points $z\in Z\setminus Z^{\TT_x}$, but $x$ generates $\TT_x$.
\end{proof}

\section{Zastava spaces and enveloping algebras: statements and fixed points analysis}\label{fixed points}

\subsection{Notation}
Let $G$ be an almost simple simply connected algebraic group over $\CC$. We denote by $\calB$ the flag variety of
$G$. Let us also fix a pair of opposite Borel subgroups $B$, $B_-$ whose intersection is a maximal torus $T$ (thus we have
$\calB=G/B=G/B_-$). We denote by $\grg$ the Lie algebra of $G$ and by ${\mathbf t}$, $\grb$, $\grb_-$ the Lie algebras of $T$, $B$, $B_-$ respectively.

Let $\Lam$ denote the cocharacter lattice of $T$; since $G$ is assumed to be simply connected, this is also the coroot lattice of $G$.
We denote by $\Lam_+\subset \Lam$ the sub-semigroup spanned by positive coroots. We say that $\alp\geq \beta$ (for $\alp,\beta\in \Lam$)
if $\alp-\beta\in\Lam_+$. We also let $\rho$ denote the half-sum of the positive coroots of $\grg$.

It is well-known that $H_2(\calB,\ZZ)=\Lam$ and that an element $\alp\in H_2(\calB,\ZZ)$ is representable by an algebraic curve
if and only if $\alp\in \Lam_+$.

Let $\grg^\svee$ denote the Langlands dual Lie algebra of $\grg$. It contains $\bft^*$ as its Cartan subalgebra. For any $\lam\in \bft$ we denote by $M(\lam)$ the Verma module over $\grg^\svee$ with highest weight $\lam$; we also denote by $\del M(\lam)$ the corresponding dual Verma module.
There is a unique (up to scalar) non-zero map $M(\lam)\to \del M(\lam)$. Its image is an irreducible $\grg^\svee$-module $L(\lam)$.

\subsection{The stack $\overline{\Bun}_B$ and quasi-maps}\label{bunbb}
Let $C$ be a smooth projective curve (later we are going to take $C=\PP^1$). Then we can consider the stack $\Bun_G$ of principal $G$-bundles on $C$ and the stack $\Bun_B$ of principal $B$-bundles. The stack $\Bun_B$ is disconnected; its connected components are in one-to-one correspondence with $\Lam$. For any $\theta\in \Lam$ we shall denote the corresponding connected component by $\Bun_B^{\theta}$.

We have a natural map $\Bun_B\to \Bun_G$ which is representable but not proper even when restricted to each $\Bun_B^{\theta}$. For each $\theta$ we can define certain canonical relative compactification $\overline\Bun_B^{\theta}\to \Bun_G$ of the map $\Bun_B^{\theta}\to \Bun_G$. Set also
$\overline{\Bun}_B=\cup_{\theta\in\Lam} \overline\Bun_B^{\theta}$ (cf. \cite{BG} or \cite{Bra-icm} for an expository review of the subject). A $\CC$-point of $\overline\Bun_B$ consists of the following data:

1) A $G$-bundle $\calP_G$

2) $T$-bundle $\calP_T$

3) For any dominant weight $\lam^\svee$ of $G$ an injective map of coherent sheaves $\kappa^{\lam^\svee}\colon  \calP_T^{\lam^\svee}\to \calP_G^{V(\lam^\svee)}$ . Here $V(\lam^\svee)$ denotes the irreducible representation of $G$ with highest weight $\lam^{\vee}$, $\calP_G^{V(\lam^\svee)}$ is the vector bundle on $C$ associated with $\calP_G$ via $V(\lam^\svee)$ and $\calP_T^{\lam^\svee}$ is the corresponding line bundle.

These data must satisfy the following conditions: for any two dominant weight $\lam^\svee, \mu^\svee$ we have $\kappa^{\lam^\svee+\mu^\svee}=\kappa^{\lam^\svee}\otimes \kappa^{\mu^\svee}$ (this makes sense since $V(\lam^\svee+\mu^\svee)$ is naturally a submodule of $V(\lam^\svee)\otimes V(\mu^\svee)$).

Consider the map $\Spec(\CC)\to \Bun_G$ which corresponds to the trivial $G$-bundle on $C$ and let $\theta\in \Lam$. Note that the cartesian product $\Bun_B^{\theta}\underset{\Bun_G}\times \Spec(\CC)$ is just the space of maps $C\to \calB$ of degree $\theta$. Then we define
the space of quasi-maps $C\to \calB$ of degree $\theta$ as $\overline{\Bun}_B^{\theta}\underset{\Bun_G}\times \Spec(\CC)$.

\subsection{Zastava spaces}
Let $\oZ^{\alp}$ denote the space of maps $\PP^1\to \calB$ of degree $\alp$ sending $\infty\in \PP^1$ to $\grb_-\in \calB$.
It is known~\cite{fkmm} that this is a smooth symplectic affine algebraic variety.

We denote by $Z^{\alp}$ the corresponding space of based quasi-maps of degree $\alp$ (i.e.\ quasi-maps of degree $\alp$ which have no defect at $\infty$ and such that the corresponding map sends $\infty$ to $\grb_-$). This is an affine algebraic variety. It has a natural action of the group $B$ and a commuting action of $\CC^{\times}$ (the former acts on $\calB$ preserving $B$ and the latter on $\PP^1$ preserving $\infty$).
In particular, it has an action of the torus $\TT:=T\times \CC^{\times}$.

Given $\alp,\beta\in \Lam_+$ such that $\alp\geq \beta$ we have a natural $\TT$-equivariant closed embedding
$Z^{\beta}\to Z^{\alp}$ (it is called ``adding defect at 0"). We denote by $Z$ the inductive limit of the $Z^{\alp}$'s with respect to this embedding. It is an ind-scheme endowed with an action of $\TT$.

\subsection{The setup}
Let us recall some notation already introduced in the Introduction.
Set $\calF^{\alp}=\IC(Z^{\alp}),\ \calF=\oplus_{\alp\in \Lam_+}\calF^{\alp}$.
We want to use the formalism of~\S\ref{ind-int}
for the pair $(Z,\calF)$. For this we need to describe the algebra $\calU$.

Let $U_{\hbar}(\grg^\svee)$ be the normalized universal enveloping algebra and let
$\calZ_{\hbar}(\grg^\svee)$ be its center. We then set
\[
\calU=U_{\hbar}(\grg^\svee)\underset{\calZ_{\hbar}(\grg^\svee)}\otimes S(\grt^*)
\]
as in~\S\ref{relatives}.

We are now ready to formulate the first main result of this paper:
\begin{Theorem}\label{main1}
  There exists a homomorphism of graded algebras $\iota\colon  \calU\to\Ext_{\TT}^*(\calF,\calF)$
  that satisfies the following properties:
  \begin{enumerate}
  \item The map $\iota$ is injective.
   \item Let $\theta\in \Lam$ and let $\calU_{\theta}$ be the part of $\calU$ of weight $\theta$. Then
   \[
   \iota(\calU_{\theta})\subset \prod\limits_{\alp\in \Lam_+}\Ext_{\TT}(\IC(Z^{\alp}),\IC(Z^{\alp+\theta})).
   \]
  \item
  Let $x\in \grt$ be of the form $(\mu,n)$ where $\mu\in \Lam, n\in \BN$. Then such $x$ is good in the sense of~\S\ref{ind-int}.
  \item
  For any $x$ as above the modules $\Del_{x,w,\sigma}$ and $\nabla_{x,w,\sigma}$ can only be non-zero if $\sig$ is trivial.
  \item
    For $x$ as above set $\lam=n^{-1}\mu$. Assume in addition that  $\lam$ is  integral, regular
    and dominant. Then the set
    $W_x$ is just the Weyl group $W$ of $\grg$. Moreover, in this case for any $w\in W$ we have
    $\Del_{x,w}\simeq M(w(\lam)-\rho)$ and
  $\nabla_{x,w}\simeq \del M(w(\lam)-\rho)$. More generally, for any integral dominant $\lam$ the set
  $W_x$ can naturally be identified with
  $W/\Stab_W(\lam)$, where $\Stab_W(\lam)$ denotes the stabilizer of $\lam$ in $W$, and the
  same description of standard and costandard modules holds true.
  \item
  Let $x=(\mu,n)$, where $\mu\in \Lam$. Assume that $\lam=n^{-1}\mu$ is
  regular and dominant (in the sense that
  $\langle\alpha^\svee,\lambda\rangle\not\in-\BN$ for any
  $\alpha^\svee\in R^{\vee+}$). Let $_\lambda W$ denote the subgroup of $W$ consisting of all $w\in W$
  such that $\lam-w(\lam)\in \Lam$. Then the set $W_x$ is equal to $_\lambda W$. Moreover,
  for any $w\in {}_\lambda W$ we have $\Del_{x,w}\simeq M(w(\lam)-\rho)$ and
  $\nabla_{x,w}\simeq \del M(w(\lam)-\rho)$. Similar statement holds in the case when $\lambda$
  is not necessarily regular (we shall not formulate it here in order not to overload the
  reader with notations). The appropriate choice of $W_x$ in the most general case is spelled out
  in~\S\ref{rati}.

\end{enumerate}
\end{Theorem}

The proof will occupy~\S\S\ref{fixed}-\ref{standard and co-standard}.

\section{Fixed points on zastava}\label{fixed}
In this Section we want to explain how Theorem \ref{main1} implies the Kazhdan-Lusztig
conjecture using the formalism of~\S\ref{abs-form}.

\subsection{Integral case: the setup}\label{nota}
For a coweight $\lambda\in \Lam$
consider a one-parametric subgroup
$\TT_{\lam}:=\{(\lambda(c),c),\ c\in\CC^\times \}\subset\TT$ (the graph of
$\lambda$). Recall that the torus $\TT$ acts on the zastava space $Z^\alpha$.
We are interested in the
fixed point set $(Z^\alpha)^{\TT_{\lam}}$. In principle, the case of $\lambda$ is dominant is of principal importance to us, but for the sake of completeness and for future applications we shall treat the non-dominant case too. Recall that
$\oZ^\alpha\subset Z^\alpha$ stands for the moduli space of
degree $\alpha$ based maps from $\BP^1$ to the flag variety $\CB$.
The $T$-fixed points in $\CB$ are numbered by the Weyl group $W$ so that
the neutral element $e\in W$ corresponds to $\grb_-\in\CB$, and the longest
element $w_0\in W$ corresponds to $\grb\in\CB$. We denote the orbit
$U_-\cdot w:=\on{Rad}B_-\cdot w$ (resp.\ $U\cdot w:=\on{Rad}B\cdot w$) by $\CB^w$ (resp.\ $\CB_w$).
The intersection $\CB^w_y:=\CB^w\cap\CB_y$ is nonempty if and only if $w\geq y$.
By a slight abuse of notation, we will denote the intersection of $\CB_y$
with the Schubert variety $\overline\CB{}^w$ by $\overline\CB{}^w_y$.

For a subset $J\subset I$ of simple roots we denote by $L_J$ the corresponding
Levi subgroup (so that $L_\emptyset=T$). We denote by $W_J$ the parabolic subgroup of $W$
generated by the simple reflections in $J$. The corresponding generalized flag
variety $G/P_J$ is denoted $^J\CB$, and the flag variety of $L_J$ is denoted
$_J\CB$. We have a natural projection
$^J\CB\stackrel{\pi_J}{\longleftarrow}\CB$ with a fiber $_J\CB$.
The $T$-fixed points in $^J\CB$
are numbered by the right cosets $W/W_J=:W^J$.
\begin{NB}
  Did we define $W_J$ ?
  \begin{NB2}
    Done.
    \end{NB2}%
\end{NB}%
As above, we have $U_\pm$-orbits,
their intersections and closures $^J\overline\CB{}^w_y$ and so on. Note that
the open cell $^J\CB_e$ is the orbit of $\pi_J(\grb_-)$ under the free action of
the unipotent radical $^JU:=\on{Rad}P_J$.

\subsection{Regular dominant case}\label{reg dom}
First let $\lambda$ be regular dominant.
\begin{Lemma}\label{phix}
  If there is $w\in W$ such that $\lambda-w\lambda=\alpha$, then
$(\oZ^\alpha)^{\TT_{\lam}}\simeq\CB^w\cap\CB_e$. Otherwise,
$(\oZ^\alpha)^{\TT_{\lam}}$ is empty.
\end{Lemma}

\begin{proof} Take a point
$z\in\CB_e=U\cdot \grb_-$ (the big Schubert cell centered at $\grb_-\in\CB$ in the
flag variety). Consider a map $\CC^\times \to\CB,\ c\mapsto\lambda(c^{-1})\cdot z$.
It extends uniquely to a map
$\CC^\times \subset\BP^1\stackrel{\phi_z}{\longrightarrow}\CB$. Clearly, $\phi_z$
is $\TT_{\lam}$-invariant, and $\phi_z(\infty)=\grb_-$ (since $\lambda(\CC^\times )$
contracts $U\cdot \grb_-$ to $\grb_-$). Conversely, given a $\TT_{\lam}$-invariant
map $\phi\colon \BP^1\to\CB$ such that $\phi(\infty)=\grb_-$, we set
$z_\phi:=\phi(1)\in U\cdot \grb_-$.

Clearly, $\phi(0)$ is a $\lambda(\CC^\times )$-fixed point (equivalently, $T$-fixed
point) $w\in W\subset\CB$. The argument in the proof of~\cite[Lemma~5.3]{BFI}
shows that $\deg\phi=\lambda-w\lambda\in\Lambda_+=\BN[I]$. Finally, for
$\phi=\phi_z$, we have $\phi(0)=w$ if and only if $z\in\CB^w$.
\end{proof}

We now want to study the fixed points of $\TT_{\lam}$ in $Z^{\alp}$. In what follows
$(Z^{\alp})^{\TT_{\lam}}$ will always denote the corresponding \emph{reduced} scheme.
\begin{Corollary}\label{zas} There exists a morphism
\[
(Z^\alpha)^{\TT_{\lam}}\to
\bigcup\limits_{w: \lambda-w\lambda\leq\alpha}\overline\CB{}^w_e
\]
(the latter is understood
as a locally closed subvariety of $\calB$) which is bijective on $\CC$-points.
Moreover, assume that $\alp$ is such that the subvariety
$\bigcup_{w: \lambda-w\lambda\leq\alpha}\overline\CB{}^w_e$ of $\calB$ is
irreducible (in particular, this is true when $\alp$ is sufficiently large).
Then the above morphism is an isomorphism.\footnote{The statement is probably
true for any $\alp$ but at this point we don't know a proof.}
\end{Corollary}

\begin{proof} A $\TT_{\lam}$-invariant based quasimap can have a defect only at $0\in\BP^1$.
So the evaluation at $1\in\BP^1$ is a well-defined morphism
$p\colon (Z^\alpha)^{\TT_{\lam}}\to\CB$. The proof of Lemma \ref{phix} shows that the image
of $p$ coincides with $\bigcup_{w: \lambda-w\lambda\leq\alpha}\overline\CB{}^w_e$,
and $p$ is bijective onto its image. For $\alpha$ such that
$\bigcup_{w: \lambda-w\lambda\leq\alpha}\overline\CB{}^w_e$ is irreducible (and hence
open in a Schubert variety $\overline\CB{}^y$ for some $y\in W$), the
normality of $\overline\CB{}^y$ and reducedness of $(Z^\alpha)^{\TT_{\lam}}$
finishes the proof.
\end{proof}
\begin{Corollary} The (reduced) scheme $Z^{\TT_{\lam}}$ is isomorphic to $\calB_e$. The stratification induced on $\calB_e$ by the standard stratification on $Z$ is that by $\calB^w_e$'s. In particular, this proves the first assertion of part (4) of Theorem \ref{main1}.
\end{Corollary}

\subsection{Regular non-dominant case} Now let $\lambda$ be regular, and $y\in W$ be such that
$\lambda':=y^{-1}\lambda$ is (regular) dominant.

\begin{Proposition}\label{reg zas}
There is a morphism $(Z^\alpha)^{\TT_{\lam}}\to\bigcup\limits_{w: \lambda'-w\lambda'\leq\alpha}
\overline\CB{}^w_y$ which defines a bijection on $\CC$-points. It is an isomorphism if the RHS is irreducible.
\end{Proposition}

\begin{proof} For $\phi\in(Z^\alpha)^{\TT_{\lam}}$ we consider $\dot y{}^{-1}\cdot\phi$
(for a representative $\dot y$ of $y$ in the normalizer of $T$),
following the proof of~\cite[Theorem~5.2 before~Lemma~5.3]{BFI}.
Now we repeat the argument in the proof of Corollary~\ref{zas}.
\end{proof}

\subsection{Singular dominant}\label{irreg dom}
Now let $\lambda$ be singular dominant, and let
$J:=\{j\in I: \langle\alpha^\svee_j,\lambda\rangle=0\}$.

\begin{Lemma}\label{psix}
If there is $w\in W^J$ such that $\lambda-w\lambda=\alpha$, then
$(\oZ^\alpha)^{\TT_{\lam}}\simeq\ ^J\CB^w_e$. Otherwise,
$(\oZ^\alpha)^{\TT_{\lam}}$ is empty.
\end{Lemma}

\begin{proof} We repeat the argument of the proof of Lemma \ref{phix}, and note that
the attractor of $\grb_-$ under the action of $\lambda(\CC^\times )$ is nothing but
the orbit of the radical $^JU\cdot \grb_-$. The projection $\pi_J$ takes this
orbit isomorphically onto $^J\CB_e$. The fixed point set $\CB^{\TT_{\lam}}$ is
nothing but the preimage $\pi_J^{-1}(W^J)$ of the fixed point set
$^J\CB^{\TT_{\lam}}=W^J$. For $w\in W^J$, we denote by $\CR^w$ the repellent of
the connected component $\pi_J^{-1}(w)$ under the action of $\lambda(\CC^\times )$.
Then $\pi_J$ takes $\CR^w\cap(\ ^JU\cdot \grb_-)$ isomorphically onto its
image $^J\CB^w_e$.
\end{proof}

The same argument as in the proof of Corollary~\ref{zas} establishes
\begin{Corollary}\label{zasr}
  There exists a morphism
\[
(Z^\alpha)^{\TT_{\lam}}\to\bigcup\limits_{w\in W^J: \lambda-w\lambda\leq\alpha}\
^J\overline\CB{}^w_e
\]
(the latter is understood
as a locally closed subvariety of $^J\calB$) which is bijective on $\CC$-points.
Moreover, assume that $\alp$ is such that the subvariety
$\bigcup\limits_{w\in W^J: \lambda-w\lambda\leq\alpha}\
^J\overline\CB{}^w_e$ of $^J\calB$ is
irreducible (in particular, this is true when $\alp$ is sufficiently large).
Then the above morphism is an isomorphism.
\end{Corollary}
\subsection{Singular}\label{irreg}
Now let $\lambda$ be singular of the form
$y\lambda'$ where $\lambda'$ is (singular) dominant. The coweight $\lambda'$
defines a subset $J\subset I$ as in~\S\ref{irreg dom}, and the class of $y$
in $W^J$ is well defined. The same argument as in the proof of Proposition \ref{reg zas}
establishes
\begin{Proposition}\label{irreg zas}
  There exists a morphism
\[
(Z^\alpha)^{\TT_{\lam}}\to\bigcup_{w\in W^J: \lambda'-w\lambda'\leq\alpha}\
 ^J\overline\CB{}^w_y
\]
(the latter is understood
as a locally closed subvariety of $^J\calB$) which is bijective on $\CC$-points.
Moreover, assume that $\alp$ is such that the subvariety
$\bigcup_{w\in W^J: \lambda'-w\lambda'\leq\alpha}\
 ^J\overline\CB{}^w_y$ of $^J\calB$ is
irreducible (in particular, this is true when $\alp$ is sufficiently large).
Then the above morphism is an isomorphism.
\end{Proposition}

\subsection{Rational case}\label{rat}
Now suppose that our 1-parametric subgroup in $\TT$ projects onto
$\CC^\times $ not bijectively, but with degree $n$. We choose a coordinate on this
subtorus, and view its projection into $T$ as a cocharacter $\mu\in \Lambda$.
We define $\lambda:=\mu/n\in \Lambda\otimes\BQ$, and continue to call our
1-parametric subtorus $\TT_{\lam}$. We choose a primitive $\sqrt[n]{1}$, and
set $\zeta:=\mu(\sqrt[n]{1})$. The centralizer $G_\zeta\subset G$ of
$\zeta\in T\subset G$ is connected reductive since $G$ is simply-connected
(a Borel-Siebenthal subgroup); $T$ is a maximal torus of $G_\zeta$, and
the corresponding Weyl group $_\zeta W$ coincides with the subgroup $_\lambda W$ of $W$
formed by all $w\in W$ such that $\lambda-w\lambda\in\Lambda$.

\medskip
\noindent
{\bf Remark.} In general such a subgroup needs not be a Coxeter group, but for
simply connected $G$ it is. The set of its Coxeter generators $I_\zeta$ is
numbered by the indecomposable elements of
$R^{\vee+}_\zeta:=\{\alpha^\svee\in R^{\vee+}:\
\langle\alpha^\svee,\lambda\rangle\in\BZ\}$.

The flag variety $_\zeta\CB$ of $G_\zeta$ embeds into $\CB$, and we keep the notation
of~\S\ref{nota} for $_\zeta\CB$ like $_\zeta\overline\CB{}^w_y$ etc.

\subsection{Rational regular dominant}\label{rat reg dom} Assume that $\lambda$ is
regular dominant (in the sense that
$\langle\alpha^\svee,\lambda\rangle\not\in-\BN$ for any
$\alpha^\svee\in R^{\vee+}$).

\begin{Lemma}\label{r phix} If there is $w\in {}_\zeta W$ such that $\lambda-w\lambda=\alpha$,
then $(\oZ^\alpha)^{\TT_{\lam}}\simeq {}_\zeta\CB^w_e$. Otherwise,
$(\oZ^\alpha)^{\TT_{\lam}}$ is empty.
\end{Lemma}

\begin{proof}
For $z\in\CB_e=U\cdot \grb_-$ we consider a map
$\CC^\times \to\CB,\ c\mapsto\mu(c^{-1})\cdot z$. This map factors through the $n$-fold
covering $\CC^\times \to\CC^\times ,\ c\mapsto c^n$ if and only if $z\in (U\cap G_\zeta)\cdot \grb_-=\
_\zeta\CB_e$. If it does, we denote the resulting map from the base of the
$n$-fold covering by $\CC^\times \to\CB\colon c\mapsto\lambda(c^{-1})\cdot z$. Now we
continue the same way as in the proof of Lemma \ref{phix}.
\end{proof}

The same argument as in the proof of Corollary \ref{zas} establishes
\begin{Corollary}\label{rat zas}
  There exists a morphism
\[
(Z^\alpha)^{\TT_{\lam}}\to\bigcup\limits_{w\in\, _\zeta W: \lambda-w\lambda\leq\alpha}\
_\zeta\overline\CB{}^w_e
\]
(the latter is understood
as a locally closed subvariety of $_\zeta\calB$) which is bijective on $\CC$-points.
Moreover, assume that $\alp$ is such that the subvariety
$\bigcup\limits_{w\in\, _\zeta W: \lambda-w\lambda\leq\alpha}\
_\zeta\overline\CB{}^w_e$ of $_\zeta\calB$ is
irreducible (in particular, this is true when $\alp$ is sufficiently large).
Then the above morphism is an isomorphism.
\end{Corollary}

\subsection{Rational regular}\label{rat reg} Now let $\lambda$ be rational regular of the
form $y\lambda'$ where $\lambda'$ is regular dominant in the sense
of~\S\ref{rat reg dom}, and $y\in {}_\zeta W$. The standard by now argument
establishes
\begin{Proposition}\label{reg rat zas}
There exists a morphism
\[
(Z^\alpha)^{\TT_{\lam}}\to\bigcup_{w\in {}_\zeta W:\
\lambda'-w\lambda'\leq\alpha} {}_\zeta\overline\CB{}^w_y
\]
(the latter is understood
as a locally closed subvariety of $_\zeta\calB$) which is bijective on $\CC$-points.
Moreover, assume that $\alp$ is such that the subvariety
$\bigcup_{w\in {}_\zeta W:\
\lambda'-w\lambda'\leq\alpha} {}_\zeta\overline\CB{}^w_y$ of $_\zeta\calB$ is
irreducible (in particular, this is true when $\alp$ is sufficiently large).
Then the above morphism is an isomorphism.
In particular, $Z^{\TT_{\lam}}$ is isomorphic to $_\zeta \calB_y$.
\end{Proposition}

\subsection{Rational singular dominant}\label{rat sing dom} Assume $\lambda$ is rational
singular dominant, and let $J_\zeta:=\{j\in I_\zeta: \langle\alpha^\svee_j,
\lambda\rangle=0\}$. The same argument as in the proof of Lemma \ref{psix}
and Corollary \ref{zasr} establishes
\begin{Corollary}\label{zasrat}
  There exists a morphism
\[
(Z^\alpha)^{\TT_{\lam}}\to\bigcup_{w\in {}_\zeta W^{J_\zeta}: \lambda-w\lambda\leq\alpha}\
_{\hphantom{J}\zeta}^{J_\zeta}\overline\CB{}^w_e
\]
(the latter is understood
as a locally closed subvariety of $_\zeta^{J_\zeta}\calB$) which is bijective on $\CC$-points.
Moreover, assume that $\alp$ is such that the subvariety
$\bigcup_{w\in {}_\zeta W^{J_\zeta}: \lambda-w\lambda\leq\alpha}\
_{\hphantom{J}\zeta}^{J_\zeta}\overline\CB{}^w_e$ of $_\zeta^{J_\zeta}\calB$ is
irreducible (in particular, this is true when $\alp$ is sufficiently large).
Then the above morphism is an isomorphism.
\end{Corollary}
\subsection{Rational}\label{rati}
Finally, let $\lambda$ be rational singular of the form $y\lambda'$ where
$\lambda'$ is (singular) dominant. The coweight $\lambda'$ defines a subset
$J_\zeta\subset I_\zeta$ as in~\S\ref{rat sing dom}, and the class of $y$
in $_\zeta W^{J_\zeta}$ is well defined. The standard by now argument establishes
\begin{Proposition}\label{rat sing}
There exists a morphism
\[
(Z^\alpha)^{\TT_{\lam}}\to\bigcup\limits_{w\in {}_\zeta W^{J_\zeta}: \lambda'-w\lambda'\leq\alpha}\
_{\hphantom{J}\zeta}^{J_\zeta}\overline\CB{}^w_y
\]
(the latter is understood
as a locally closed subvariety of $_\zeta^{J_\zeta}\calB$) which is bijective on $\CC$-points.
Moreover, assume that $\alp$ is such that the subvariety
$\bigcup\limits_{w\in {}_\zeta W^{J_\zeta}: \lambda'-w\lambda'\leq\alpha}\
_{\hphantom{J}\zeta}^{J_\zeta}\overline\CB{}^w_y$ of $_\zeta^{J_\zeta}\calB$ is
irreducible (in particular, this is true when $\alp$ is sufficiently large).
Then the above morphism is an isomorphism.
\end{Proposition}

\subsection{Kazhdan-Lusztig conjecture}\label{kl}
We now assume Theorem \ref{main1}. Then using the above results and the results of Section \ref{abs-form} we get the following result.
For any $w,y\in {}_\zeta W^{J_\zeta}$, the multiplicity of $L(y(\lambda)-\rho)$
in $M(w(\lambda)-\rho)$ equals the total costalk dimension
$\dim i^!_w\operatorname{IC}( {}_{\hphantom{J}\zeta}^{J_\zeta}\overline\CB{}^w_y)$.
This result is well known, and follows from~\cite[Theorem~1.24]{reductive},
\cite[Corollary~5.35]{jantzen}, and~\cite[Theorem~3.11.4]{BGS}.

\section{Construction of the action}\label{action}

In this Section we define the action of the algebra $\calU$ on the sheaf $\calF$ (cf.\ the notation of Theorem \ref{main1}).
The construction is an easy corollary of the derived Satake equivalence (cf.\ \cite{BF}) combined with some results from \cite{BG}.

\subsection{Derived Satake equivalence}
As usually we set $\calO=\CC[\![t]\!],\ \calK=\CC(\!(t)\!)$. We let $\Gr_G=G(\calK)/G(\calO)$ denote the affine Grassmannian of $G$. This is an ind-projective ind-scheme over $\CC$ of ind-finite type which carries a natural action of $G(\calK)\rtimes \CC^\times$. For a pro-algebraic subgroup $H$ of $G(\calO)\rtimes \CC^\times$ we denote by $D_H(\Gr_G)$ the ind-completion of the $H$-equivariant derived category of constructible sheaves on $\Gr_G$.\footnote{More precisely,
we choose one of the standard dg-enhancements of $D^b_H(\Gr_G)$ (see e.g.~\S\ref{examples}(b) below),
take its ind-completion, and then pass to its homotopy category. Note that $D_H(\Gr_G)$
is {\em not} equivalent to the unbounded $H$-equivariant derived constructuble category of
$\Gr_G$. It is often called the {\em renormalized} $H$-equivariant derived category.}
Let $\Perv_H(\Gr_G)$ denote the corresponding category of perverse sheaves.
It is well-known that $\Perv_{G(\calO)}(\Gr_G)=\Perv_{G(\calO)\rtimes \CC^\times}(\Gr_G)\simeq \Rep(G^{\vee})$. For a dominant weight $\lam$ of $G^{\vee}$ we denote by $\IC^{\lam}$ the corresponding object of $\Perv_{G(\calO)}(\Gr_G)$; it is the IC-complex of the closure of a $G(\calO)$-orbit on $\Gr_G$. This category naturally acts on $D_H(\Gr_G)$ for any $H$ (by convolutions on the right).

Let now $A$ be a graded $\CC$-algebra. We denote by $A^{[]}$ the corresponding dg-algebra with zero differential. We denote by $A^{[]}$-mod the category of dg-modules over $A^{[]}$ and by $D(A^{[]})$ the corresponding derived category.
Let also $D_{\perf}(A^{[]})$ denote the category of perfect complexes.
By definition this is the full subcategory of $D(A^{[]})$ generated by $A^{[]}$ by means of shifts, direct sums, cones and direct summands. If a reductive algebraic group $H$ acts on $A$ then we denote by $D^H(A^{[]})$ (resp.\ $D^H_{\perf}(A^{[]})$) the corresponding derived category of $H$-equivariant dg-modules (resp.\ the derived category of perfect $H$-equivariant modules), i.e.\ modules endowed with an action of $H$ compatible with its action on $A$. This category is endowed with a natural action of the tensor category $\Rep(H)$.

Consider now the category $D_{\perf}^{G^{\vee}}(U^{[]}_{\hbar}(\grg^\svee))$. This category has a natural monoidal structure.
In order to construct it, let us just remark that the structure of a $U_{\hbar}(\grg^\svee)$-module
endowed with a compatible (with respect to adjoint action) $G^{\vee}$-action on a vector space $M$ is the same as the structure of
a $U_{\hbar}(\grg^\svee)$-bimodule endowed with an action of $G^{\vee}$ such that the action of the diagonal copy of $\hbar\grg^\svee$ is equal to the derivative of the $G^{\vee}$-action multiplied by $\hbar$ (we shall call such gadgets \emph{asymptotic Harish-Chandra bimodules over} $\grg^\svee$). Given two asymptotic Harish-Chandra bimodules $M,N$ their tensor product $M\underset{U_{\hbar}(\grg^\svee)}\otimes N$ is again an asymptotic
Harish-Chandra bimodule. Hence the category of asymptotic
Harish-Chandra bimodules has a natural monoidal structure. It is easy to upgrade this construction to dg-modules; in this way we obtain a monoidal structure on $D_{\perf}^{G^{\vee}}(U^{[]}_{\hbar}(\grg^\svee))$ (cf.~\cite{BF} for more details). The same applies to the
bigger category $D^{G^\vee}(U^{[]}_{\hbar}(\grg^\svee))$.

The following result is proved in~\cite{BF}:

\begin{Theorem}\label{Satake}
\begin{enumerate}
\item
  Let $D^b_{G(\calO)\rtimes \CC^{\times}}(\Gr_G)$ denote the bounded derived category of
  $G(\calO)\rtimes \CC^{\times}$-equivariant sheaves on $\Gr_G$ with constructible cohomology. Then
there exists an equivalence of mo\-no\-idal triangulated categories
\[\Psi_G\colon  D^b_{G(\calO)\rtimes \CC^{\times}}(\Gr_G)\tilde{\to} D_{\perf}^{G^{\vee}}(U^{[]}_{\hbar}(\grg^\svee))\] Here $G^{\vee}$ acts on $U_{\hbar}(\grg^\svee)$ by means of the adjoint action.

\item
The functor $\Psi_G$ sends the action of $\Perv_{G(\calO)}(\Gr_G)$ on the LHS to the natural $\Rep(G^{\vee})$-action on the RHS. In particular, we have $\Psi_G(\IC^{\lambda})=U_{\hbar}^{[]}(\grg^\svee)\otimes L(\lambda)$ where $L(\lambda)$ stands for the irreducible representation of $\grg^\svee$ with highest $\lambda$ and the action of $G^{\vee}$ on $U_{\hbar}^{[]}(\grg^\svee)\otimes L(\lambda)$ is the tensor product of the adjoint action of $G^{\vee}$ on $U_{\hbar}^{[]}(\grg^\svee)$ and its action on $V(\lam)$.

\item
Let   $D_{G(\calO)\rtimes \CC^{\times}}(\Gr_G)$ denote the ind-completion of the category $D^b_{G(\calO)\rtimes \CC^{\times}}(\Gr_G)$.
The functor $\Psi_G$ extends to an equivalence between $D_{G(\calO)\rtimes \CC^{\times}}(\Gr_G)$ and $D^{G^{\vee}}(U^{[]}_{\hbar}(\grg^\svee))$.
\end{enumerate}
\end{Theorem}
\begin{Remark}The category $D_{G(\calO)\rtimes \CC^{\times}}(\Gr_G)$ is the renormalized derived category in the terminology of \cite{ArGa}; note that it is not the same as the actual derived category of $G(\calO)\rtimes \CC^{\times}$-equivariant sheaves on $\Gr_G$ (the latter is studied in \cite{ArGa} as well (without the loop rotation) cf.~\cite[Theorem 12.5.3]{ArGa} for its description in Langlands dual terms.
\end{Remark}
\subsection{De-equivariantization}
Let $\calC$ be a co-complete additive category with an action of the tensor category
$\Rep(G^{\vee})$ (here $G^{\vee}$ can be any algebraic group). For
$X\in \calC,\ V\in \Rep(G^{\vee})$ we shall denote the action of $V$ on $X$ by $V\star X$.
Then we can construct
another category $\calC_{G^{\vee}}$ endowed with an action of $G^{\vee}$ (cf.~\cite{AG}),
called the \emph{de-equivariantization} of $\calC$. We also refer the reader
to~\cite[Examples 12 and 21]{Ga05} for the definition of the relevant notions; note that in
{\em loc.cit.\ }the categories in question are assumed to be abelian but the definitions are the same for arbitrary additive categories and in any case we shall soon switch to dg-categories for which the relevant theory is developed in~\cite{1-aff}.

By definition, an object of $\calC_{G^{\vee}}$ is an object $X$ of $\calC$ endowed with an action of $R_{G^{\vee}}$ where
$R_{G^{\vee}}$ is the regular representation of $G^{\vee}$ considered as a ring-object in $\Rep(G^{\vee})$.
\begin{NB}
  It means that we have a homomorphism $R_{G^\vee}\star X\to X$.
\end{NB}%
More explicitly, an object of $\calC_{G^{\vee}}$ is an object $X$ of $\calC$ together with a system
of morphisms $V\star X\to \unl{V}\otimes X$
for any $V\in \Rep(G^{\vee})$ (here $\unl{V}\otimes X$ just stands for the tensor product of the
vector space $\unl{V}$ underlying a $G$-module $V$ and an object $X$),
satisfying some natural compatibilities with respect to direct sums and tensor
products.\footnote{It is not difficult to see that these compatibilities imply that above
  morphisms are necessarily isomorphisms, but we shall not use this.}

The category $\calC_{G^{\vee}}$ is endowed with an  action of the group $G$ and it makes sense to
consider the category $\calC^{G^{\vee}}$ of $G^{\vee}$-equivariant objects.

By definition we have a natural forgetful functor $\sI_{G^{\vee}}\colon \calC_{G^{\vee}}\to \calC$
(we denote this functor by $\sI$ since in certain more general situation it is related to the
functor of induction from a subgroup). This functor admits a left adjoint $\sF_{G^{\vee}}$ which
sends an object $X$ to $R_{G^{\vee}}\star X$. In the case when $\calC=\Rep(G^{\vee})$ we have
$\calC_{G^{\vee}}=\Vect$ and this functor is the natural forgetful functor
from $\Rep(G^{\vee})$ to $\Vect$.

The following result is proved in \cite{AG} in the abelian case and in \cite{1-aff} in the
dg-case, cf.~\cite[Theorem 2.2.2]{1-aff}:
\begin{Theorem}
\begin{enumerate}
\item For abelian categories the procedures of equivaritization and de-equivarintizaion are mutually inverse.
\item The same is true for dg-categories (i.e.\ these procedures define a pair of mutually inverse equivalences between the $(\infty,1)$-category of dg-categories with weak $G^{\vee}$-action and the $(\infty,1)$-category of dg-categories with $\Rep(G^{\vee})$-action).
\end{enumerate}
\end{Theorem}

Some remarks are in order. First, let us say that in the terminology of~\cite{1-aff} the notion
of category with (weak) $G^{\vee}$-action is the same as the notion of sheaf of categories over
the classifying stack $\on{pt}\!/G^{\vee}$. The corresponding category of equivariant objects is the
category of global sections of this sheaf on $\on{pt}\!/G^{\vee}$. Thus assertion (2) above is indeed
equivalent to~\cite[Theorem 2.2.2]{1-aff}.

The second remark is this: given a dg-category $\calC$ one can consider its homotopy category
$\Ho(\calC)$. Now assume that $\calC$ is endowed with a $\Rep(G^{\vee})$-action. Then typically
$\Ho(\calC_{G^{\vee}})$ is not the same as $\Ho(\calC)_{G^{\vee}}$, and certainly the ``correct" object
to work with is the former rather than the latter. Thus from now on we shall switch to working
with dg-categories. So if we want to de-equivariantize a triangulated category $\calD$, we must
do it on dg-level. In other words, we first must choose a dg-model (or dg-lift) of $\calD$
(i.e.\ a pre-triangulated dg-category $\calD^{\on{dg}}$ such that $\Ho(\calD^{\on{dg}})\simeq \calD$)
also endowed with $\Rep(G^{\vee})$-action and then consider
$\Ho(\calD^{\on{dg}}_{G^{\vee}})$. This category might depend on the choice of the dg-lift. However,
abusing slightly the notation, when the choice of the dg-lift is clear we shall denote the
resulting category just by $\calD_{G^{\vee}}$.

\subsection{Examples}
\label{examples}
\textup{(a)}
Let $A$ be a dg-algebra over $\CC$ with a $G^{\vee}$-action. Let $\calD$ be the derived category of $G^{\vee}$-equivariant dg-modules over $A$.
Its natural dg lift $\calD^{\on{dg}}$ is the full subcategory of the category of $G^{\vee}$-equivariant dg-modules over $A$ generated by modules of the form $A\otimes V$ where $V$ is a finite-dimensional representation of $G^{\vee}$ by means of shifts, direct summands, cones and (infinite) direct sums.
The de-equivariantized category $\calD_{G^{\vee}}$ is then just the derived category of dg-modules over $A$.

\textup{(b)}
Let now $\calD$ be the derived Satake category $D_{G(\calO)\rtimes \CC^{\times}}(\Gr_G)$. Then its natural dg-lift can be constructed in 3 steps:

1) We consider the dg-category $(\calD^{\on{dg}})'$ whose objects are perverse $G(\calO)$-equivariant sheaves on $\Gr_G$ and
$\Hom_{\calD'}(\calF,\calG)=\oplus_i \Ext^i_{G(\calO)}(\calF,\calG)$ (the RHS is viewed as a complex with trivial differential).

2) We consider the Bondal-Kapranov pre-triangularisation $(\calD^{\on{dg}})'_{\text{pretr}}$, see
e.g.~\cite[\S1.5.4]{BV}) (i.e.\ we formally add direct summands, shifts and cones).

3) We let $\calD$ be the ind-completion of $(\calD^{\on{dg}})'_{\text{pretr}}$ (since we want to work with co-complete categories).

It follows from the very construction of \cite{BF} that the equivalence of Theorem \ref{Satake}(3) extends to an equivalence of the corresponding dg-lifts. Hence we have $(D_{G(\calO)\rtimes \CC^{\times}}(\Gr_G))_{G^{\vee}}\simeq D(U^{[]}_{\hbar}(\grg^{\vee}))$.

The main result of this Section is a construction of a functor $\Phi_Z$ as in the theorem below:
\begin{Theorem}\label{functor-z}
There exists a natural triangulated functor
\[
\Phi_Z\colon  D_{G(\CO)\rtimes\CC^{\times}}(\Gr_G)_{G^{\vee}}\to D_{\TT}(Z)
\]
such that
\begin{enumerate}
  \item
  The object $\sF_{G^{\vee}}(\IC^0)$ goes over to $\calF=\oplus_{\alp} \IC(Z^{\alp})$.
  \item
    Given an object $\calG\in D_{G(\CO)\rtimes\CC^{\times}}(\Gr_G)_{G^{\vee}}$,
    the following diagram commutes:
    \begin{equation*}
    \begin{CD}
      \calZ_{\hbar}(\grg^\svee)\simeq H^*_{G(\CO)\rtimes\CC^{\times}}(\on{pt})
      @>>> \Ext^*_{D_{G(\CO)\rtimes\CC^{\times}}(\Gr_G)_{G^{\vee}}}(\calG,\calG) \\
      @VVV @VV{\Phi_Z}V\\
      S(\grt) \simeq H^*_\TT(\on{pt}) @>>> \Ext^*_{\TT}(\Phi_Z(\calG),\Phi_Z(\calG)).
    \end{CD}
  \end{equation*}
\end{enumerate}
\end{Theorem}






\noindent The proof of~Theorem~\ref{functor-z} occupies~\S\S\ref{more an}-\ref{end cons}.

It follows from Theorem \ref{Satake} that $\Ext^*(\sF_{G^{\vee}}(\IC^0),\sF_{G^{\vee}}(\IC^0))$ (computed in the category $D_{G(\CO)\rtimes\CC^{\times}}(\Gr_G)_{G^{\vee}}$) is equal to $U_{\hbar}(\grg^\svee)$. Hence the first assertion of Theorem \ref{functor-z}
gives rise to a homomorphism $U_{\hbar}(\grg^\svee)\to \Ext^*_{\TT}(\calF,\calF)$. The second assertion guarantees it extends to a homomorphism
$\calU\to \Ext^*_{\TT}(\calF,\calF)$ which is claimed in Theorem \ref{main1}.

  \subsection{More abstract nonsense}
  \label{more an}
In order to construct the functor $\Phi_Z$ we need one more abstract categorical construction. Assume that in the setup of the previous subsection $\calC$ is actually a monoidal category. Assume in addition that we have a monoidal functor $S\colon \Rep(G^{\vee})\to \calC$ so that the action of $V\in \Rep(G^{\vee})$ on $\calC$ is the right multiplication by $S(V)$.\linelabel{action_line}
Hence we have a system of functorial isomorphisms $X\star S(V) \simeq \underline{V}\otimes X$ for
$X\in\mathcal C_{G^\vee}$, $V\in\operatorname{Rep}(G^\vee)$.
\begin{NB}
  If I denote by $\otimes^\calC$ the tensor product in $\calC$, we require
  $V\star M \cong M\otimes^\calC S(V)$.

  It is probably not a good idea to introduce a new notation
  $\otimes^\calC$. So how about adding the following at the end of
  l.\ref{action_line} ?

  `Hence we have a system of functorial isomorphisms $X\star S(V) \simeq \underline{V}\otimes X$ for $X\in\mathcal C_{G^\vee}$, $V\in\operatorname{Rep}(G^\vee)$.'
  \begin{NB2}
    Done.
    \end{NB2}%
\end{NB}%

Let us now assume that we are given another additive category $\calD$ endowed with an action
of the monoidal category $\calC$. Then any $M\in \on{Ob}(\calD)$ defines a functor
$\Act_M\colon \calC\to \calD$ which sends any $X\in \on{Ob}(\calC)$ to $X\star M$.

\begin{Lemma}
  Assume that for an object $M\in \on{Ob}(\calD)$ we are given a system of functorial
isomorphisms
\begin{equation}\label{objectM}
S(V)\star M\simeq \unl{V}\otimes M\ \text{for $V\in \Rep(G^{\vee})$}
\end{equation}
%
%
satisfying obvious compatibility conditions with respect to tensor products.
Then the functor $\Act_M$ can be upgraded to a functor\footnote{Here we use the action of
  $\Rep(G^{\vee})$ on $\calC$ by means of right multiplication.}
\[\Act_{M,G^{\vee}}\colon \calC_{G^{\vee}}\to \calD.\]
\end{Lemma}

\begin{proof}
The idea is as follows. Let $X\in \on{Ob}(\calC_{G^{\vee}})$.
The object $\sI_{G^{\vee}}(X)\star M$ of $\calD$ has a natural structure of a module over the ring $R_{G^{\vee}}$. Indeed, to construct this module structure it is enough for any $V\in \Rep(G^{\vee})$ to construct
an isomorphism $\unl{V}\otimes (\sI_{G^{\vee}}(X)\star M)\simeq \unl{V}\otimes (\sI_{G^{\vee}}(X)\star M)$.
The latter is constructed as the composition
\[
\begin{aligned}
  \unl{V}\otimes (\sI_{G^{\vee}}(X)\star M)=\sI_{G^{\vee}}(X)\star(\unl{V}\otimes M)\simeq \sI_{G^{\vee}}(X)\star
  (S(V)\star M)\simeq \\
  (\sI_{G^{\vee}}(X)\star S(V))\star M\simeq (\unl{V}\otimes \sI_{G^{\vee}}(X))\star M=
  \unl{V}\otimes (\sI_{G^{\vee}}(X)\star M).
\end{aligned}
\]
\begin{NB}
  The second term of the second line should be
  $(\unl{V}\otimes \sI_{G^\vee}(X))\star M$. The last $\simeq$ should be $=$.
  \begin{NB2}
    Done.
    \end{NB2}%
\end{NB}%
We would like to define $\Act_{M,G^{\vee}}(X)$ as the fiber
of $\sI_{G^{\vee}}(X)\star M$ at $1\in G^{\vee}$ (i.e.\ take tensor product with $\CC$ over $R_{G^{\vee}}$
where the latter acts on the former via the evaluation map at $1\in G^\vee$).
Here we again have to work on dg-level. Namely, we assume that

1) The categories $\calC$ and $\calD$ are triangulated categories which are homotopy categories of some co-complete pre-triangulated dg-categories $\calC^{\on{dg}},\calD^{\on{dg}}$.

2) All of the above functors (monoidal structure on $\calC$, action of $\calC$ on $\calD$, the functor $\Rep(G^{\vee})\to \calC$ and the corresponding morphisms between them) are lifted to the dg-level.

\noindent
Then the above construction makes sense and we indeed get a functor $\Act_{M,G^{\vee}}\colon \calC_{G^{\vee}}\to \calD$. The above choice of a dg-model will be automatic in applications (since all the categories involved will be some kind of derived categories of sheaves).
\end{proof}

\subsection{The stack $_{c,\infty}\overline\Bun_B$}\label{bg}
Let $C$ again be a smooth projective curve and let $c\in C$.
Recall the stack $_{c,\infty}\overline\Bun_B$ introduced in~\cite[4.1.1]{BG}.
It classifies the same data as in~\S\ref{bunbb} except that now the maps $\kappa^{\lam^\svee}$ are allowed to have arbitrary poles at $c$.
It contains $\overline\Bun_B$ as a closed sub-stack.
More generally, for any $\alp\in\Lambda$ the ind-stack $_{c,\infty}\overline\Bun_B$ contains a locally closed sub-stack $_{c,\alpha}\overline\Bun_B$  such that the closure
\[
_{c,\geq\alpha}\overline\Bun_B:=\overline{_{c,\alpha}\overline\Bun_B}
=\bigsqcup\limits_{\beta\geq \alpha}{_{c,\beta}\overline\Bun_B}.
\]
By definition $_{c,\alpha}\overline\Bun_B$ consists of all triples $(\calP_G,\calP_T,\kappa)$ such that for any dominant weight $\lam^\svee$  the map
$\calP_T^{\lambda^\svee}(\langle \alp,\lam^\svee\rangle\cdot x)\to \calP_G^{V(\lambda^\svee)}$ has neither zeroes nor poles at $c$.

We have the Hecke correspondences
\[
_{c,\infty}\overline\Bun_B\stackrel{'h^\leftarrow_G}{\longleftarrow} {}_{c,\infty}\Hecke_{B,G,c} \stackrel{'h^\rightarrow_G}{\longrightarrow} {}_{c,\infty}\overline\Bun_B,
\]
 see~\cite[4.1.4]{BG}.
 Let us also take $C=\PP^1,\ c=0$. Then the above correspondences
give rise to the convolution action of the equivariant derived category $D_{G(\CO)\rtimes\CC^{\times}}(\Gr_G)$
on the equivariant derived category of $D_{\CC^{\times}}(_{c,\infty}\overline\Bun_B)$.
Let now $_{c,\infty}\overline\Bun_B'$ denote the ind-stack classifying the same data as before together with a trivialization of $\calP_G$ at $\infty\in \PP^1$ such that

1) The generalized $B$-structure $\kappa$ does not have a defect at $\infty$;

2) The value of the $B$-structure $\kappa$ at $\infty\in \PP^1$ is equal to $B_-$ (this makes sense because the fiber of $\calP_G$ at $\infty$ is trivialized).

\noindent
This stack is endowed with a natural action of the torus $\TT=T\times \CC^{\times}$ where the 2nd factor just acts on $\PP^1$ (preserving the points $0$ and $\infty$) and the 1st factor acts by changing the trivilization of $\calP_G$ at $\infty$ (note that the action of $T$ on $\calB$ preserves the point $B_-$).

\noindent
Similarly we can define stacks $_{c,\geq \alpha}\overline\Bun_B'$.
In particular, $_{c,\geq0}\overline\Bun_B'=:\overline\Bun_B'$.
Then in the same way we get an action of $D_{G(\CO)\rtimes\CC^{\times}}(\Gr_G)$ on
$D_{\TT}(_{c,\infty}\overline\Bun_B')$.

In addition the ind-stacks $_{c,\infty}\overline\Bun_B$ and $_{c,\infty}\overline\Bun_B'$ admit a natural action of the lattice $\Lambda$ (by modifying the $T$-bundle $\calP_T$ at $c$). Hence the corresponding derived categories of sheaves also admit an action of $\Lambda$ which is the same as the action of the tensor category $\Rep(T^{\vee})$.
We would like now to use the formalism of the previous subsection. Namely, we set $\calC=D_{G(\CO)\rtimes\CC^{\times}}(\Gr_G)$,
$\calD=D_{\TT}(_{c,\infty}\overline\Bun_B')_{T^{\vee}}$, $M=\sF_{T^{\vee}}(\IC(\overline\Bun_B'))$.
\begin{NB}
  Could you confirm that this is \emph{not}
  $\sF_{T^\vee}(\IC(_{c,\infty}\overline\Bun_B'))$ ? $\overline\Bun_B'$
  is viewed as a closed substack of
  $_{c,\infty}\overline\Bun_B'$. (This is not explained yet, as
  we only explained the corresponding statement without ${}'$.
  \begin{NB2}
    I do confirm. The IC-sheaf of the ind-stack $_{c,\infty}\overline\Bun_B'$ is not defined at
    our present level of education. $\overline\Bun_B'$ is now introduced in the previous paragraph.
    \end{NB2}%
\end{NB}%
Note that the de-equivariantized derived category $D_{\TT}(_{c,\infty}\overline\Bun_B')_{T^{\vee}}$ is just the category of $\Lambda$-equivariant objects in the usual derived category $D_{\TT}(_{c,\infty}\overline\Bun_B')$. In this realization we have
\[
M=\bigoplus\limits_{\alpha\in \Lambda} \IC(_{c,\geq \alpha}\overline\Bun_B').
\]
The isomorphisms (\ref{objectM}) are essentially given
by ~\cite[Theorem 4.1.5]{BG}.\footnote{Formally in \cite{BG} the $\CC^{\times}$-equivariant situation was never considered, but since all the sheaves involved in (\ref{objectM}) are perverse, the results of ~\cite[Theorem 4.1.5]{BG} extend automatically to the equivariant situation.}
Thus we get a functor $\Act_{M,G^{\vee}}\colon D_{G(\CO)\rtimes\CC^{\times}}(\Gr_G)_{G^{\vee}}\to
D_\TT(_{c,\infty}\overline\Bun_B')_{T^\vee}$.
\begin{NB}
  $D_{\CC^{\times}}(_{c,\infty}\overline\Bun_B')_{T^\vee}$ ?
  a) ${}'$ is missing, b) $T^\vee$ is missing.
  \begin{NB2}
    Done.
    \end{NB2}%
\end{NB}%

\subsection{End of the construction}
\label{end cons}
To finish the construction of $\Phi_Z$ it remains to notice that the ind-scheme $Z$ is an open sub-ind-scheme of the ind-stack $_{c,\infty}\overline\Bun_B'$ (given by the condition that the $G$-bundle $\calP_G$ is trivial). So we define the functor $\Phi_Z$ to be
the composition of $\Act_{M,G^{\vee}}$ with the (open) restriction to $Z$, followed by the forgetful functor $D_{\TT}(Z)_{T^{\vee}}\to D_{\TT}(Z)$.
The fact that it satisfies the conditions of Theorem \ref{functor-z} is straightforward.


\subsection{Factorization}
Recall the factorization morphism $\pi_\theta\colon Z^\theta\to\AA^\theta$. It is defined as the
pullback of the Schubert divisor (the union of $B$-invariant divisors)
in the flag variety $\CB$.
Recall the factorization property of $\pi$: a canonical isomorphism for a decomposition
$\theta=\alpha+\eta$
\[(Z^\alpha\times Z^\eta)|_{(\AA^\alpha\times\AA^\eta)_{\on{disj}}}\cong
Z^\theta\times_{\AA^\theta}(\AA^\alpha\times\AA^\eta)_{\on{disj}}.\]
We may further restrict this isomorphism to the open subset
\[\oZ^\alpha\times Z^\eta|_{(\BG_m^\alpha\times\AA^\eta)_{\on{disj}}}\subset
(Z^\alpha\times Z^\eta)|_{(\AA^\alpha\times\AA^\eta)_{\on{disj}}}.\]
Given $\calF_1,\calF_2\in D_\TT(Z)$ and $e\in\Ext^*_{D_\TT(Z)}(\CF_1,\CF_2)$ we can restrict
$e$ to $Z^\theta$, pull it back to $Z^\theta\times_{\AA^\theta}(\AA^\alpha\times\AA^\eta)_{\on{disj}}$,
and then restrict to the open subset $\oZ^\alpha\times Z^\eta|_{(\BG_m^\alpha\times\AA^\eta)_{\on{disj}}}\subset
Z^\theta\times_{\AA^\theta}(\AA^\alpha\times\AA^\eta)_{\on{disj}}$ to obtain $e_{\alpha,\eta}$.
In particular, by definition, $e_{0,\theta}=e|_{Z^\theta}$.

\begin{Lemma}
  \label{tautolog}
  \textup{(a)} Given $\calP\in D_{G(\calO)\rtimes\CC^\times}(\Gr_G)_{G^\vee},\ \Phi_Z(\calP)$ is
  factorizable in the sense that $\Phi_Z(\calP)$ restricted to
  $\oZ^\alpha\times Z^\eta|_{(\BG_m^\alpha\times\AA^\eta)_{\on{disj}}}$ is canonically isomorphic to
  $\big(\IC(\oZ^\alpha)\boxtimes\Phi_Z(\calP)|_{Z^\eta}\big)|_{(\BG_m^\alpha\times\AA^\eta)_{\on{disj}}}$.

  \textup{(b)} Given $e\in\Ext^*_{D_{G(\calO)\rtimes\CC^\times}(\Gr_G)_{G^\vee}}(\calP_1,\calP_2),\ \Phi_Z(e)$
  is factorizable in the sense that $\Phi_Z(e)_{\alpha,\eta}$
  is equal to $\on{Id}_{\IC(\oZ^\alpha)}\boxtimes\Phi_Z(e)_{0,\eta}$.
\end{Lemma}
\begin{proof}
First of all, it is clear that (b) follows from (a) (provided that the natural isomorphism in the formulation of (a) is functorial with respect to $\calP$). On the other hand (a) is essentially the statement of~\cite[Proposition 5.1.5]{abbgm} (in {\em loc.cit.} it is assumed that $\calP$ is perverse but this does not affect the proof).
\end{proof}

\section{Standard and co-standard modules}\label{standard}
\label{standard and co-standard}
In this Section we finish the proof of Theorem \ref{main1}. Part (2) of the Theorem is obvious from the construction of the homomorphism $\iota$; part (1) will follow from the computation of the action of $\calU$ on the stalk of $\calF$ at $0$. Part (3) immediately follows from parts (4), (5) and (6). So, the main work will be to prove parts (4), (5) and (6). In fact, the set $W_x$ was already computed
in~\S\ref{fixed}, so we just need to identify standard and co-standard modules with Verma modules and their duals.

\subsection{The costalk at $0$}
\label{the costalk}
We begin by the following. Let $\grs\colon  \on{pt}\to Z$ be the embedding of the point $0$
(the unique $\TT$-fixed point).
Then $\grs^!\calF$ can naturally be viewed as an object of $D_{\TT}(\on{pt})$ and the
$S(\grt^*)$-module $H^*_{\TT}(\grs^!\calF)$  acquires a natural structure of a
graded $\calU$-module.
We want to give its algebraic description. Before we do this, let us first recall the computation
of its graded character~\cite[Theorem 1.12]{BFGM}. Namely,
$\on{ch}(\grs^!\calF)=\sum_{\alpha\in\Lambda_+}K_\alpha(q)$, where $\deg q=2$, and
$K_\alpha(q)=\sum_{\alpha_1+\ldots+\alpha_i=\alpha}q^i$ (the sum is taken over the set of unordered
partitions of $\alpha$ into a sum of positive roots of $\grg^\svee$).

\subsection{The module $\calM$}
Now, let us define algebraically a graded module $\calM$ over $\calU$ (which will eventually
turn out to be isomorphic to $H^*_{\TT}(\grs^!\calF)$).
Namely, let us first set
\[
\calM'=U_{\hbar}(\grg^\svee)/\grn_+^\svee.
\]
It has a natural structure of a left $U_{\hbar}(\grg^\svee)$-module.
In addition it has the following structures:

1) Right action of $S(\grt^*)\simeq S_{\hbar}(\bft^*)$;

2) Action of the group $T^\vee$ (i.e.\ grading by $\Lambda$).

\noindent
These structures are compatible in the following way: given $a\in \bft^*$ the action of $a$ on
$\calM'$ coming from deriving the action of the
group $T^\vee$ multiplied by $\hbar$ is equal to the difference of the left action of
$a\in U_{\hbar}(\grg^\svee)$ and the right action of $a\in S_{\hbar}(\bft^*)$.

In what follows we are going to slightly change the right action of $S(\grt^*)$.
Namely, we consider the following $S(\grt^*)-S(\grt^*)$-bimodule $S(\grt^*)(-\rho)$:
as a left $S(\grt^*)$-module it coincides with the free rank one module $S(\grt^*)$, but
the right action of $a\in\bft^*$ is multiplication by $a-\hbar\langle a,\rho\rangle$.
Finally, $\calM:=\calM'\otimes_{S(\grt^*)}S(\grt^*)(-\rho)$.

\begin{Theorem}
  \label{universal Verma}
$H^*_{\TT}(\grs^!\calF)$ is isomorphic to $\calM$ as a graded $\calU$-module.
\end{Theorem}

Let us remark that Theorem \ref{universal Verma} implies part (1) of Theorem \ref{main1}, since the action of the algebra $\calU$ on $\calM$ is faithful.
\begin{proof}
   We consider the $S(\grt^*)$-submodule $\calH$ of $H^*_{\TT}(\grs^!\calF)$
  spanned by $1\in\IC(Z^0)$ (notation $\calH$ stands for the highest weight component).
  So $\calH$ is a right $S(\grt^*)$-module, and it carries a
  commuting left $S(\grt^*)$-action from the embedding $S(\grt^*)\subset\calU$.
  First we identify the $S(\grt^*)-S(\grt^*)$-bimodule $\calH$ with $S(\grt^*)(-\rho)$.

  The right action of $S(\grt^*)$ is nothing but the action of
  $H^*_{T(\calO)\rtimes\CC^\times}(\on{pt})$ arising from the $T(\calO)\rtimes\CC^\times$-equivariance
  of $\IC(_{c,\geq0}\overline\Bun_B)$. On the highest weight component this right action arises
  from the $T\times\CC^\times$-equivariance of the skyscraper sheaf
  $\IC(Z^0)$. The left action of $S(\grt^*)$ is the action
  of $H^*_{T(\calO)\rtimes\CC^\times}(\on{pt})$ arising from
  $T(\calO)\rtimes\CC^\times\subset G(\calO)\rtimes\CC^\times$-equivariance of the skyscraper sheaf
  $\IC^0\in D_{G(\calO)\rtimes\CC^\times}(\Gr_G)$ and the canonical isomorphism
  $\IC^0\star\IC(_{c,\geq0}\overline\Bun_B)\cong\IC(_{c,\geq0}\overline\Bun_B)$.
  Also, the left action of $\calZ_\hbar(\grg^\svee)=S(\grt^*)^W\subset S(\grt^*)$ is nothing
  but the action of $H^*_{G(\calO)\rtimes\CC^\times}(\on{pt})$ arising from the
  $G(\calO)\rtimes\CC^\times$-equivariance of $\IC^0$. We conclude that the left action
    of $\calZ_\hbar(\grg^\svee)=S(\grt^*)^W$ coincides with the restriction of the right action
    of $S(\grt^*)$ to the $W$-invariants. Since the left action of $\calZ_\hbar(\grg^\svee)$
    is the action of the Harish-Chandra center of $U_\hbar^{[]}(\grg^\svee)$, it follows that
    the right action of $S(\grt^*)$ is shifted by $-\rho$.

  From the universality of the Verma module $\calM$ we have a morphism
  $\Xi\colon\calM\to H^*_{\TT}(\grs^!\calF)$ of graded $\calU$-modules taking $1\in\calM$ to
  $1\in\IC(Z^0)$. The source is a free $S(\grt^*)$-module, and $\Xi$ is generically injective,
  since the Verma module with generic highest weight is irreducible. Hence $\Xi$ is injective.
  The comparison of graded characters
  (see~\S\ref{the costalk}) implies that $\Xi$ is an isomorphism.
\end{proof}

Recall the setup of~\S\ref{general case}. Consider $x=(\mu,t)\in\grt$ such that
$\mu\in{\mathbf t},\ t\ne0$. We consider the hyperbolic restriction
$\calF_x=\Phi_x(\calF)$,
and its costalk $\grs^!\calF_x$. Then $H^*_{\TT}(\grs^!\calF_x)$
acquires a natural structure of a graded $\calU$-module, as well as its fiber
$\big(H^*_{\TT}(\grs^!\calF_x)\big)_x$ at $x\in\grt$.

\begin{Corollary}
  \label{special fiber}
  The $U(\grg^\svee)$-module $\big(H^*_{\TT}(\grs^!\calF_x)\big)_x$ is isomorphic
  to the Verma module $M(t^{-1}\mu-\rho)$.
\end{Corollary}

\begin{proof}
  The $U(\grg^\svee)$-module $\big(H^*_{\TT}(\grs^!\calF)\big)_x$ is isomorphic
  to the Verma module $M(t^{-1}\mu-\rho)$ by~Theorem~\ref{universal Verma}.
  The natural morphism $H^*_{\TT}(\grs^!\calF)\to H^*_{\TT}(\grs^!\calF_x)$
  induces an isomorphism of fibers at $x\in\grt$ by~Lemma~\ref{standard-other}.
\end{proof}

\subsection{Proof of Theorem~\ref{main1}(5,6)}
By Lemma~\ref{standard-other} we have to study $(H^*_{\TT_x}(\phi,s_{x,w}^!\calF))_x$, where
$s_{x,w}$ is the embedding into $Z$ of a fixed point $\phi$ in the stratum $Z^{\TT_x}_w$.
The analysis of the dual $(H^*_{\TT_x}(\phi,s_{x,w}^*\calF))_x$ is entirely similar and will be
omitted.
By Lemma~\ref{phix} applied in the case of integral regular dominant $\lambda$
(resp.~Lemma~\ref{psix} in the case of integral singular dominant $\lambda$,
resp.~Lemma~\ref{r phix} in the case of rational regular dominant $\lambda$), $\phi\in\oZ^\alpha$,
where $\alpha=\lambda-w\lambda$. More precisely, $\oZ^\alpha$ is embedded into $Z$ by adding
defect at $0\in\AA^1$. Let us fix $\Lambda_+\ni\theta\geq\alpha$. Then the image $\phi_\theta$
of $\phi$ in $Z^\theta$ is $\phi_\alpha(\eta\cdot0)$, where $\phi_\alpha\in\oZ^\alpha$,
and $\Lambda_+\ni\eta:=\theta-\alpha$.

We will need the shifted universal Verma module $\calM(-\alpha)$ defined as follows.
We consider the following $S(\grt^*)-S(\grt^*)$-bimodule $S(\grt^*)(-\alpha)$:
as a left $S(\grt^*)$-module it coincides with the free rank one module $S(\grt^*)$, but
the right action of $a\in\bft^*$ is multiplication by $a-\hbar\langle a,\alpha\rangle$.
We set $\calM(-\alpha):=\calM\otimes_{S(\grt^*)}S(\grt^*)(-\alpha)$.
Furthermore, we set $\calM(-\alpha)_{\grt_x}:=\calM(-\alpha)\otimes_{S(\grt^*)}S(\grt_x^*)$:
a ``partially universal'' shifted Verma module.

Since the highest weight of the $\calU$-module $H^*_{\TT_x}(\phi,s_{x,w}^!\calF)$ is
$-\rho-\nobreak\alpha$, from the ``partial universality'' of $\calM(-\alpha)_{\grt_x}$
we obtain a morphism
$\Xi\colon \calM(-\alpha)_{\grt_x}\to H^*_{\TT_x}(\phi,s_{x,w}^!\calF)$.
\begin{Claim} $\Xi$ is an isomorphism. \end{Claim}
Indeed, both LHS and RHS are free $S(\grt_x^*)$-modules graded by
$-\Lambda_+$, finitely generated in each degree. Hence it suffices to check that the fiber
$\Xi_0$ of $\Xi$ at $0\in\grt_x$ is an isomorphism. But this is a statement about
\emph{nonequivariant} cohomology.

Recall the factorization morphism $\pi_\theta\colon Z^\theta\to\AA^\theta$. It is defined as the
pullback of the Schubert divisor (the union of $B$-invariant divisors)
in the flag variety $\CB$. In particular, we have $\pi_\alpha(\phi_\alpha)=\alpha\cdot0$.
Now instead of $B$, we consider a general Borel subgroup $B'$ transversal to $B_-$.
It gives rise to a new factorization morphism $\pi'_\theta\colon Z^\theta\to\AA^\theta$
(defined as the pullback of the $B'$-Schubert divisor). For a general $B'$ the divisor
$\pi'_\alpha(\phi_\alpha)$ is disjoint from $0\in\AA^1$, and we fix such a $B'$.
Recall the factorization property of $\pi'$: a canonical isomorphism
\[(Z^\alpha\times Z^\eta)|_{(\AA^\alpha\times\AA^\eta)_{\on{disj}}}\cong
Z^\theta\times_{\AA^\theta}(\AA^\alpha\times\AA^\eta)_{\on{disj}}.\]
It follows that $s_{x,w}^!\IC(Z^\theta)\cong s_{\phi_\alpha}^!\IC(Z^\alpha)\otimes\grs^!\IC(Z^\eta)$,
where $s_{\phi_\alpha}$ denotes the embedding of the point $\phi_\alpha$ into $Z^\alpha$.
Now $\phi_\alpha\in\oZ^\alpha\subset Z^\alpha$ is a smooth point, so the first factor
$s_{\phi_\alpha}^!\IC(Z^\alpha)$ is trivial: $\BC$ in cohomological degree $2|\alpha|$.
The second factor $\grs^!\IC(Z^\eta)$ was studied 
in~Theorem~\ref{universal Verma}.
So we obtain an isomorphism $M(-\alpha-\rho)\iso H^*(\phi,s_{x,w}^!\calF)$
that coincides with $\Xi_0$ by~Lemma~\ref{tautolog}. So the claim is proved.

\medskip

Having established that $\Xi$ is an isomorphism, we are ultimately interested in its fiber
at $x=(\mu,n)\in\grt_x$. Recall that $\alpha=\lambda-w\lambda$, and this is
the shift in the highest weight we have to apply. Also, in the present case $t=n$, and
$t^{-1}\mu=n^{-1}\mu=\lambda$.
We get $\lambda-\rho-(\lambda-w\lambda)=w\lambda-\rho$, and Theorem~\ref{main1}(5,6) is proved.
\hfill $\Box$

\section{Affine case} \label{affine case}
\subsection{More notation}
In this Section we would like to describe a conjectural extension of
the above results to affine Lie algebras.  We again would like to follow the formalism
of~\S\ref{abs-form}.
Let $G,\grg$ be as above and let $\grg_{\aff}$ denote the corresponding affine Kac-Moody Lie algebra.
By definition it is the canonical central extension of the algebra $\grg[t,t^{-1}]\rtimes \CC$
where $\CC$ acts on $\grg[t,t^{-1}]$ by derivative of the loop rotation.
Let also $\grg_{\aff}^{\vee}$ denote the Langlands dual affine Kac-Moody algebra (i.e.\ this
the Kac-Moody algebra whose root system is dual to that of $\grg_{\aff}$).

We let $\widehat{T}=T\times \BC^\times$. This torus naturally sits inside
$G[t,t^{-1}]\rtimes \CC^\times$ where $\BC^\times$ acts by loop rotations on the affine
objects.
Furthermore, $\widehat{I}:=I\sqcup i_0$ (the affine simple root);
$\widehat{W}=W\ltimes X_*(T)$ is the affine Weyl group. Surprisingly, the
affine flag varieties stay hatless. Instead, $\CB^w$ (resp. $\CB_w$) stands for
a Schubert cell in the ind-finite type affine flag variety $\calB$ (resp.\ Kashiwara
affine flag scheme $\underline\CB$). We introduce two new notations: $\underline\CB{}^w_y$
is the intersection of $\CB^w$ with the closure of $\CB_y$, and $\overline\CB^w_y$ is the
intersection of $\CB_y$ with the closure of $\CB^w$. In particular,
$\underline\CB^w_e=\CB^w$. In other words,
\begin{equation*}
    \underline\CB^w_y = \bigsqcup_{v\ge y} \CB^w\cap \CB_v,
    \qquad
    \overline\CB^w_y = \bigsqcup_{v\le w} \CB^v\cap \CB_y,
\end{equation*}
cf.~\S\ref{nota}.

\subsection{Affine zastava spaces} The affine zastava space is denoted
$Z_{\on{aff}}^\alpha$. This is an affine variety which contains the space of based maps from $\PP^1$
to $\underline\CB$ as a dense open subset. As before for every $\beta\geq \alpha$ we have a
closed embedding $Z_{\on{aff}}^\alpha\hookrightarrow Z_{\on{aff}}^\beta$ by adding defect at $0$.
More generally, for $K\subset I\subset\widehat{I}$ we denote by $Z_{\on{aff},K}^\alpha$
the corresponding parabolic affine zastava space (this is a partial compactification
of the space of based maps from $\PP^1$ to the corresponding thick partial affine flag variety
$^K\underline{\CB}$ defined by $K$).
Here $\alpha\in\Lambda^{G_{\on{aff}},K}_+$. We will be mostly interested in the
case $K=I$ where $\Lambda^{G_{\on{aff}},I}_+=\BN$, and $Z_{\on{aff},I}^a=\CU^a$
is the Uhlenbeck space.

\subsection{Summary}\label{summary}
Let us summarize what we can do in the affine case. Let $Z_{\aff}=\underset{\to}\lim~ Z^{\alpha}_{\aff}$,
$\calF=\bigoplus\limits_{\alpha} \calF^{\alp}$ where $\calF^{\alp}$ denotes the IC-sheaf of $Z_{\aff}^\alp$.
The ind-scheme $Z_{\aff}$ is endowed with an action of the torus
$\TT_{\aff}=\CC^\times\times \widehat{T}$ and the perverse sheaf $\calF_{\aff}$ is naturally
$\TT_{\aff}$-equivariant. Then we would like to do the following:

\begin{enumerate}
\item
 Define a homomorphism $\iota_{\aff}\colon U_{\hbar}(\grg_{\aff}^{\vee})\to \Ext_{\TT_{\aff}}(\calF,\calF)$
satisfying the affine version of Theorem \ref{main1}(2).

\item Compute the (ind)scheme $(Z_{\aff})^{\TT_{\aff,\lam}}$ for a subtorus
$\TT_{\aff,\lam}\subset \TT_{\aff}$ as in~\S\ref{fixed}.

\item Prove an analog of Theorem \ref{main1}.

\end{enumerate}

\noindent
Unfortunately, the technique of~\S\ref{action} is not available in the affine case, since an affine analog of the derived geometric Satake equivalence is not known at present. We believe that instead it should be possible to define $\iota_{\aff}$ by describing explicitly the images of the Chevalley generators. So far, we haven't done it carefully and we hope to address this issue in a future publication.

Let us assume that $\iota_{\aff}$ has been defined. Then, we need to perform part 2 of the above problem --- in the way which is analogous to~\S\ref{fixed}. This is what the bulk of this Section is devoted to; we believe that after the analog of parts (3)--(6) of~Theorem~\ref{main1} will be straightforward along the lines of~\S\ref{standard}.

Let us briefly explain our results about fixed points (part 2 of the above program) and what it
has to do with the affine version of the Kazhdan-Lusztig conjecture. Recall that the Lie algebra
$\grg_{\aff}^\svee$ contains canonical one-dimensional center and its weights have level which can
be positive, negative, or critical (these notions are carefully defined below; as before to
simplify the discussion we shall only consider rational weights). Let us for the sake of the
current discussion assume that $\grg$ is simply laced, which implies that $\grg_{\aff}^\svee$
is isomorphic to $\grg_{\aff}$.

It is known after Kashiwara and Tanisaki~\cite{KTpos,KTneg} that geometrically category $\calO$
for $\grg_{\aff}$ on positive level is equivalent to certain category of $D$-modules on the
Kashiwara flag scheme $\underline{\calB}$ and on negative level --- on the ind-scheme $\calB$.
Thus geometrically multiplicities of Verma modules in simple modules (i.e.\ coefficients of the
expansion of classes of simple modules in the basis of Verma modules in the K-group of category
$\calO$) are given by dimensions of stalks of IC-sheaves of Schubert varieties in
$\underline{\calB}$ for representations of positive level and by dimensions of stalks of
IC-sheaves of Schubert varieties in $\calB$ for negative level. For critical level it is
believed that similar results hold for ``semi-infinite Schubert varieties" (whose IC stalks
can be defined using zastava spaces $Z^\alpha$ --- cf.~\cite{Bra-icm}); we do not know a reference
to such a result,
although it should be easy to deduce it from the results of Frenkel and Gaitsgory~\cite{FrGa}.

Recall that in the finite-dimensional case, we have reproved the Kazhdan-Lusztig conjecture in a
different formulation which differs from the standard one by Koszul duality: namely we expressed
the multiplicities of simple modules in Verma modules for the Lie algebra $\grg^\svee$ using
stalks of IC sheaves of Schubert varieties in the flag variety of $\grg$ (or its endoscopic subalgebras). Similar thing is
going to happen in the affine case: we show below that
when $\lambda$ has positive level, the variety of fixed points of $\TT_{\aff,\lambda}$ in
$Z_{\aff}^{\alp}$ has to do with Schubert varieties in $\calB$ and when $\lambda$ has negative
level the above variety has to do with Schubert varieties in $\underline{\calB}$. Thus assuming
parts 1) and 3) above we shall be able to show that multiplicities of simple modules in
Verma modules for representations $\grg_{\aff}^{\vee}$ of positive (resp. negative) level are given
by dimensions of IC-stalks of Schubert varieties in $\calB$ (resp. in $\underline{\calB}$).
The relevant Koszul duality is obtained by combination of the works of Soergel~\cite{Soergel}
and Bezrukavnikov-Yun~\cite{BY}. Similarly, assuming 1) and 3) our results on fixed points
will imply that the corresponding multiplicities on critical level are given by stalks of
IC-sheaves of semi-infinite Schubert varieties (which are given by Lusztig's periodic
polynomials~\cite{Lu3}).

Since part (1) above is only conjectural at the moment, in this paper we shall only perform the
fixed points analysis (part (2)) and we shall postpone the rest until further publication.
\begin{NB}
  You mean until the next life, right?
  \end{NB}%
The fixed points analysis is interesting in itself, since it shows how transversal slices in
various flag varieties (thin, thick or
semi-infinite) (i.e.\ intersections of Schubert varieties in such a flag variety
with the Schubert cells in the opposite flag variety)
arise in a uniform way as fixed points in the same affine zastava spaces
with respect to different tori $\TT_{\aff,\lam}$.

\subsection{$\scW$-algebras} Here is a variant of the above construction which potentially should reprove the version of the Kazhdan-Lusztig conjecture for $\scW$-algebras proved by Arakawa~\cite{arakawa}.

Let $\CU^a$ be the Uhlenbeck space.
Let $\TT = T\times\CC^\times\times\CC^\times$ as before. It acts on
$\CU^a$: $T$ acts by the change of framing, and
$\CC^\times\times\CC^\times$ acts through its natural action on the
base $\PP^2$ preserving the line at infinity. We consider equivariant
cohomology groups $H^*_{\mathbb T}(i^!\calF^a)$,
$H^*_{\mathbb T}(i^*\calF^a)$, where $i$ is the inclusion of the
unique $\TT$-fixed point of $\CU^a$, and $\calF^a = \IC(\CU^a)$.
Let $\AA = H^*_{\CC^\times\times\CC^\times}(\on{pt}) = \CC[\ve_1,\ve_2]$.
In \cite{2014arXiv1406.2381B} we constructed a representation of an
integral form of the $W$-algebra $\scW_{\AA}(\fg)$ on
$\bigoplus_a H^*_{\TT}(i^!\calF^a)$ and
$\bigoplus_a H^*_{\TT}(i^*\calF^a)$, under the assumption that $G$
is simply-laced.
Moreover the former is universal Verma module, and the latter is its
dual.
(In \cite{2014arXiv1406.2381B} we considered
${\mathbb G} = G\times\CC^\times\times\CC^\times$-equivarinat
cohomology groups, but they are just $W$-invariant parts, and the
difference is not essential.)
If these representations \emph{would} come from
$\Ext_{\TT}(\calF, \calF)$ with $\calF = \bigoplus \calF^a$, we could
apply the formalism of \S\ref{abs-form}. Hence multiplicities of
simple modules in Verma modules could be expressed by stalks of $\IC$
sheaves of Schubert varieties, and hence Kazhdan-Lusztig polynomials,
as in the case of $Z_{\aff}$ and $\grg_{\aff}^{\vee}$.
Thus this would reprove the multiplicity formula given in
\cite{arakawa}, again in the Koszul dual form.

However it is \emph{not} clear whether our construction factors
through $\Ext_{\TT}(\calF,\calF)$: $\scW_{\AA}(\fg)$ is understood as
the intersection
$$
\bigcap\limits_i \Vir_{i,\AA}\otimes_{\AA} \Heis_\AA(\alpha_i^\perp)
$$ of
tensor products of integral Virasoro and Heisenberg algebras (see
\cite[Th.~B.6.1]{2014arXiv1406.2381B}). Here $i$ runs over the index
set of simple roots of $\fg$. Individual Virasoro and Heisenberg
algebras are \emph{not} mapped to $\Ext_{\TT}(\calF,\calF)$. In fact,
they act only on $H^*_{\TT}$ of the hyperbolic restriction
$\Phi_{L_i,G}(\calF)$ (see \cite[\S4.4]{2014arXiv1406.2381B}).
Hence we need an additional care to apply our technique to this
situation. We shall postpone it until a future publication.

\begin{NB}
  If we look at the construction in \cite{2014arXiv1406.2381B} closer,
  we find that $\Vir_{i,\AA}\otimes_{\AA} \Heis_\AA(\alpha_i^\perp)$
  is mapped to
  $\Ext_{\TT}(\Phi_{L_i,G}(\calF),\Phi_{L_i,G}(\calF))$. We have a
  natural homomorphism
  \begin{equation}\label{eq:1}
    \Ext_{\TT}(\Phi_{L_i,G}(\calF),\Phi_{L_i,G}(\calF))
    \to \Ext_{\TT}(\Phi_{T,G}(\calF),\Phi_{T,G}(\calF)),
  \end{equation}
  which becomes an isomorphism over the fractional field of
  $H^*_{\TT}(\on{pt})$. (See \cite[\S3.2]{2014arXiv1406.2381B}.)
  Therefore $\scW_{\AA}(\fg)$ is mapped to the intersection of the
  images of \eqref{eq:1}. By the argument in \cite[\S3.6,
  Prop.~8.1.7]{2014arXiv1406.2381B}, we see that $\scW_{\AA}(\fg)$
  is mapped to the image of
  \begin{equation}\label{eq:2}
    \Ext_{\TT}(\calF,\calF) \to
    \Ext_{\TT}(\calF,\calF)\otimes_{H^*_{\TT}(\on{pt})}
    \on{Frac} H^*_{\TT}(\on{pt}),
  \end{equation}
  where $\on{Frac} H^*_{\TT}(\on{pt})$ is the fractional field of
  $H^*_{\TT}(\on{pt})$. If we would know that
  $\Ext_{\TT}(\calF,\calF)$ is torsion free, this is injective, hence
  we have a homomorphism
  $\scW_{\AA}(\fg) \to \Ext_{\TT}(\calF,\calF)$. However the torsion
  freeness of $\Ext_{\TT}(\calF,\calF)$ is \emph{not} known, hence we
  must be careful.

  For our application, we will study modules which are specialization
  of $H^*_{\TT}(i^!\calF)$ or $H^*_{\TT}(i^*\calF)$. The kernel of
  \eqref{eq:2}, even if it is nonzero, acts by $0$ on these modules,
  as $H^*_{\TT}(i^!\calF)$, $H^*_{\TT}(i^*\calF)$ are free over
  $H^*_{\TT}(\on{pt})$ (see
  \cite[Lemma~6.1.1]{2014arXiv1406.2381B}). Therefore we have analog
  of \thmref{main1} for the $W$-algebra.
\end{NB}%

\subsection{Twisted case} The above program has one drawback: it only deals with representations
of affine Lie algebras of the form $\grg_{\aff}^{\vee}$ --- i.e.\ affine Lie algebras which are
Langlands dual to untwisted ones.
In particular, it misses most untwisted affine Lie algebras which are associated with non-simply
laced simple finite-dimensional algebras.
\begin{NB}
  It also misses $A^{(2)}_{2n}$. As far as I understand,
  $A^{(2)}_{2n}$ is also discussed in \S\ref{twi}.
\end{NB}%
In~\S\ref{twi} we explain how to remedy this problem: namely, for any affine Lie algebra we explain
how to twist the affine zastava spaces $Z_{\aff}^{\alpha}$  so that this algebra conjecturally acts
(in the derived sense) on the corresponding direct sum of IC-sheaves (similarly to 1) above).
 We also explain how to compute the corresponding fixed points varieties (Lemma~\ref{untw}).

Let us now pass to the description of fixed points in various cases.

\subsection{Positive level}\label{posi} Let $\Lam_{\aff}=\Lam\oplus\ZZ$. It is naturally isomorphic to the cocharacter lattice of $\widehat{T}$.  Let now $\lam\in \Lam_{\aff}$. Consider the one-dimensional subtorus $\TT_{\aff,\lam}\subset \TT_{\aff}=\CC^\times\times\widehat{T}$
consisting of points of the form $(c,\lam(c))$, where $c\in \CC^\times$. The projection of
$\TT_{\aff,\lambda}\subset\BC^\times\times\widehat{T}=\BC^\times\times T\times\BC^\times$ to
$\BC^\times\times\BC^\times$ is a cocharacter given by a pair of integers $(a,b)$.
If $ab<0$ we say that the level is positive. The same argument as in~\S\ref{fixed}
establishes
\begin{Proposition} \label{aff pos}
Let $\lambda$ be a positive level rational coweight of the form $y\lambda'$
where $\lambda'$ is dominant (possibly singular). Let
$J_\zeta\subset\widehat{I}_\zeta$ be the corresponding subset, so that the
class of $y$ in $_\zeta\widehat{W}{}^{J_\zeta}$ is well defined. Then
\[
(Z_{\on{aff}}^\alpha)^{\TT_{\aff,\lambda}}\simeq
\bigcup\limits_{w\in {}_\zeta\widehat{W}{}^{J_\zeta}: \lambda'-w\lambda'\leq\alpha}\
_{\hphantom{J}\zeta}^{J_\zeta}\overline\CB{}^w_y.
\]
\end{Proposition}


\subsection{Positive level parabolic zastava}\label{posi par}
Under the assumptions
of Proposition~\ref{aff pos}, given a subset $K\subset I\subset\widehat{I}$, the image
of the composed projection $_\zeta\CB\hookrightarrow\CB\twoheadrightarrow\ ^K\CB$ is a partial flag variety
$_{\hphantom{K}\zeta}^{K_\zeta}\CB$ for certain subset
$K_\zeta\subset \widehat{I}_\zeta$.
We denote by $\bar y$ the class of $y$ in
$^{K_\zeta}_{\hphantom{K}\zeta}\widehat{W}{}^{J_\zeta}= {}_\zeta \widehat{W}_{K_\zeta}\backslash\
_\zeta\widehat{W}/ {}_\zeta\widehat{W}_{J_\zeta}$. Combining all
the above arguments we arrive at

\begin{Proposition} \label{aff pos par}
$(Z_{\on{aff},K}^\alpha)^{\TT_{\aff,\lambda}}\simeq\bigcup\limits_{w\in {}_{\hphantom{K}\zeta}^{K_\zeta}\widehat{W}{}^{J_\zeta}: \lambda'-w\lambda'\leq\alpha}\
_{\hphantom{K}\zeta}^{K_\zeta}\overline\CB{}^w_{\bar y}$.
\end{Proposition}


\subsection{Negative level} \label{nega}
 The projection of
$\TT_{\aff,\lambda}\subset\BC^\times\times\widehat{T}=\BC^\times\times T\times\BC^\times$ to
$\BC^\times\times\BC^\times$ is a cocharacter given by a pair of integers $(a,b)$.
If $ab>0$ we say that the level is negative.

\begin{Proposition} \label{aff neg}
Let $\lambda$ be a negative level rational coweight of the form $y\lambda'$
where $\lambda'$ is antidominant (possibly singular). Let
$J_\zeta\subset\widehat{I}_\zeta$ be the corresponding subset, so that the
class of $y$ in $_\zeta\widehat{W}{}^{J_\zeta}$ is well defined. Then
\[
(Z_{\on{aff}}^\alpha)^{\TT_{\aff,\lambda}}\simeq
\bigcup_{w\in {}_\zeta\widehat{W}{}^{J_\zeta}: w\lambda'-\lambda'\leq\alpha}\
_{\hphantom{j}\zeta}^{J_\zeta}\underline\CB{}_w^y.
\]
\end{Proposition}

\begin{proof}
  Invert the coordinate $z$ on the source $\BP^1$.
\end{proof}

\subsection{Negative level parabolic zastava}\label{nega par} Under the assumptions
of~Proposition~\ref{aff neg}, given a subset $K\subset I\subset\widehat{I}$, the image
of the composed projection $_\zeta\CB\hookrightarrow\CB\twoheadrightarrow\ ^K\CB$ is a partial flag variety
$_{\hphantom{K}\zeta}^{K_\zeta}\CB$ for certain subset
$K_\zeta\subset \widehat{I}_\zeta$.
We denote by $\bar y$ the class of $y$ in
$^{K_\zeta}_{\hphantom{K}\zeta}\widehat{W}{}^{J_\zeta}= {}_\zeta \widehat{W}_{K_\zeta}\backslash\
_\zeta\widehat{W}/ {}_\zeta\widehat{W}_{J_\zeta}$. Similarly
to~Proposition~\ref{aff pos par} and~Proposition~\ref{aff neg}, we have

\begin{Proposition} \label{aff neg par}
$(Z_{\on{aff},K}^\alpha)^{\TT_{\aff,\lambda}}\simeq\bigcup\limits_{w\in {}_{\hphantom{K}\zeta}^{K_\zeta}\widehat{W}{}^{J_\zeta}: w\lambda'-\lambda'\leq\alpha}\
_{\hphantom{k}\zeta}^{K_\zeta}\underline\CB{}_w^{\bar y}$.
\end{Proposition}

\subsection{Critical level} \label{crit}
We say that the level is critical if the projection
of $\TT_{\aff,\lambda}\subset\BC^\times\times\widehat{T}=\BC^\times\times T\times\BC^\times$ to
$\BC^\times\times\BC^\times$ lies entirely in the first factor $\BC^\times$. In other words,
$\TT_{\aff,\lambda}\subset\BC^\times\times T$ as in~\S\ref{fixed}. To warm up, we start with a
description of the fixed points in the Uhlenbeck space (parabolic affine zastava) $\CU^a$
(see e.g.~\cite{BFG}).
There is an involution $\tau$ acting on $\CU^a$ that preserves the action of $T$ but
interchanges the two actions of $\BC^\times$.
So applying $\tau$ we may consider the
fixed points with respect to $T'_\lambda$ whose projection to $\BC^\times\times\BC^\times$ lies entirely in the second factor $\BC^\times$. If we view $\oU^a=\on{Bun}_{G}^a(\BA^1\times\BA^1)$
as the moduli space of based maps from $\BP^1$ to the thick Grassmannian
(Kashiwara scheme) $\BG\!\on{r}_G$, then $(\oU^a)^{\TT'_\lambda}$ is
nothing but the moduli space of degree $a$ based maps from $\BP^1$ to
$\BG\!\on{r}_G^{\TT'_\lambda}=\on{Gr}_G^{\TT'_\lambda}$ (the fixed points in the usual
affine Grassmannian). As in~\S\ref{fixed}, $\lambda$ defines $\zeta\in T\subset G$,
and $G_\zeta$, and $I_\zeta$, and a dominant $_\zeta W$-conjugate $\lambda'$. Since $\lambda$ may be singular, it lies on a wall
of type $J_\zeta\subset I_\zeta$. Let $Z_{G_\zeta,J_\zeta}$ be the corresponding parabolic zastava space.

\begin{Lemma} \label{bun sec}
$(\oU^a)^{\TT'_\lambda}=
\bigsqcup\limits_{\alpha: (\alpha,\lambda')=a}\oZ_{G_\zeta, J_\zeta}^\alpha$.
\end{Lemma}

\begin{proof}
The $\TT'_\lambda$-fixed point set component passing through the base point of
$\on{Gr}_G$ is isomorphic to the loop rotation fixed point set component passing
through the point $\lambda\in\on{Gr}_{G_\zeta}$. The latter is nothing but
the parabolic flag variety $^{J_\zeta}_{\hphantom{j}\zeta}\CB$. The degrees
match because the restriction of the determinant line bundle on $\on{Gr}_G$ to
the $\TT'_\lambda$-fixed point set component  $^{J_\zeta}_{\hphantom{j}\zeta}\CB$
passing through $\lambda$ is isomorphic to $\CO_{\iota(\lambda')}$ where
$\iota\colon X_*(T)\to X^*(T)$ is the minimal even symmetric bilinear form.
\end{proof}

\begin{Corollary} \label{uhl sec}
\[
(\CU^a)^{\TT_{\aff,\lambda}}\stackrel{\tau}{\simeq}(\CU^a)^{\TT'_\lambda}=
\bigcup\limits_{\alpha: (\alpha,\lambda')=a}Z_{G_\zeta, J_\zeta}^\alpha.
\]
\end{Corollary}

\subsection{Uhlenbeck fixed points at the critical level}\label{tau}
 We offer another
explanation of the isomorphism $(\CU^a)^{\TT_{\aff,\lambda}}\simeq
\bigcup_{\alpha: (\alpha,\lambda')=a}Z_{G_\zeta, J_\zeta}^\alpha$
of~Corollary~\ref{uhl sec} without using the involution $\tau$.
To this end we view $\oU^a$ as the moduli space of based maps
from $\BP^1$ to the Beilinson-Drinfeld Grassmannian $\on{Gr}_{G,BD}$ (relative
Grassmannian over the curve $\BA^1\subset\BP^1$). Then
arguing as in the proof of~Lemma~\ref{phix} we see that
$(\oU^a)^{\TT_{\aff,\lambda}}$ is
naturally isomorphic to the locally closed subvariety of $\on{Gr}_{G,BD}$ formed
by all the collections $(\underline{z}\in\BA^{(a)},\ x\in\on{Gr}_{G,\underline{z}})$ such that
$x$ flows to the base point of $\on{Gr}_{G,\underline{z}}$ under the action of
$\lambda(\BC^\times)\subset T$ (for simplicity we
assume $\lambda$ integral here;
with obvious modifications for rational $\lambda$ like replacing $G$ with $G_\zeta$). For yet another (notational) simplification let us assume
$\lambda=\lambda'$ dominant. Then the above condition says $x\in S^0_J$
(parabolic semiinfinite orbit; more precisely, if $\underline{z}=(z_1,\ldots,z_d)$,
then $S^0_J$ stands for the product of the parabolic semiinfinite orbits
$S^0_{J,z_i}\subset\on{Gr}_{G,z_i}$ of the loop group $^JU_+(\CK)$).
The degree condition $\deg=a$ implies $x\in T^{-\alpha}_J$ (the product of
the opposite parabolic semiinfinite orbits of $^JU_-(\CK)\cdot L_J(\CO)$) for
$(\alpha,\lambda)=a$. Now the identification of $S_J^0\cap T_J^{-\alpha}$ with
$\oZ_{G_\zeta, J_\zeta}^\alpha$ is nothing but the
factorization of $\oZ_{G_\zeta, J_\zeta}^\alpha$.

\subsection{Zastava fixed points at the critical level}\label{critic}
We recall the setup of~\S\ref{crit}, but now we consider the fixed points
$(Z_{\on{aff}}^\alpha)^{\TT_{\aff,\lambda}}$. We denote by $_\zeta\CB_e(J_\zeta,y)\subset\
_\zeta\CB$ the orbit of the unipotent radical $^{J_\zeta}_{\hphantom{j}\zeta}U_+^y$
of the parabolic $_\zeta P_{J_\zeta}^y:=\dot y {}_\zeta P_{J_\zeta}\dot y{}^{-1}\subset
G_\zeta$ passing through the point $e\in {}_\zeta\CB$
(here $y\in {}_\zeta W^{J_\zeta}$). As we have seen
in~\S\ref{fixed}, conjugation by $\dot y{}^{-1}$ and projection to
$^{J_\zeta}_{\hphantom{j}\zeta}\CB$ establishes an isomorphism
$_\zeta\CB_e(J_\zeta,y)\iso\ ^{J_\zeta}_{\hphantom{j}\zeta}\CB_y$.
We denote by $_{\hphantom{j}\zeta}^{J_\zeta}\mathring{\CB Z}{}^{\alpha,w}_y$
the
locally closed subset in the product $_\zeta\CB_e(J_\zeta,y)
\times\oZ^\alpha_{G_\zeta,J_\zeta}$ formed by the pairs
$(b,\phi)$ such that the relative position of $b$ and $\phi(0)$ is
$w\in {}_\zeta W^{J_\zeta}$. Also,
$\delta_\zeta\in\Lambda_+^{G_{\zeta,\on{aff}}}\subset\Lambda_+^{G_{\on{aff}}}$
stands for the minimal imaginary root of $G_{\zeta,\on{aff}}$.

\begin{Lemma} \label{crit zas}
\[
(\oZ_{\on{aff}}^\beta)^{\TT_{\aff,\lambda}}\simeq
\bigsqcup\limits_{w\in {}_\zeta W^{J_\zeta},\ \alpha\in\Lambda_+^{G_\zeta}:~ \lambda'-w\lambda'+
  (\alpha,\lambda')\delta_\zeta=\beta}\
_{\hphantom{j}\zeta}^{J_\zeta}\mathring{\CB Z}{}^{\alpha,w}_y.
\]
\end{Lemma}

\begin{proof}
We view $\oZ_{\on{aff}}^\beta$ as the moduli space of based
maps from $\BP^1$ to the Beilinson-Drinfeld-Kottwitz flag variety
$\on{Fl}_{G,BDK}$ (relative over the curve $\BA^1\subset\BP^1$)
which projects to the Beilinson-Drinfeld Grassmannian $\on{Gr}_{G,BD}$.
Arguing as in~\S\ref{tau}, we obtain a projection $p_2$ from
$(\oZ_{\on{aff}}^\beta)^{\TT_{\aff,\lambda}}$ to
$\oZ^\alpha_{G_\zeta,J_\zeta}$.
To obtain a projection $p_1\colon (\oZ_{\on{aff}}^\beta)^{\TT_{\aff,\lambda}}\to\
_\zeta\CB_e(J_\zeta,y)$ note that for $\phi\in
(\oZ_{\on{aff}}^\beta)^{\TT_{\aff,\lambda}}$, the value $\phi(1)$
lies in the preimage of $S_0\subset\on{Gr}_{G_\zeta}$ under the natural
projection $\on{Fl}_G\to\on{Gr}_G$.
This preimage is $T$-equivariantly isomorphic to $\CB\times S_0$, and we define
$p_1(\phi)$ as the image of $\phi(1)$ under the projection to $\CB$.
The argument as in~Proposition~\ref{rat sing} proves that the image of $p_1$ lies in
$_\zeta\CB_e(J_\zeta,y)\subset\CB$.
It remains to identify the image under $(p_1,p_2)$ of the connected component of
$(\oZ_{\on{aff}}^\beta)^{\TT_{\aff,\lambda}}$ containing a $\TT_{\aff,\lambda}$-fixed
point $\phi$ such that $\phi(0)=\widehat{w}=
-\alpha\cdot w\in\Lambda\rtimes W=\widehat{W}$.

To this end note first that the latter condition guarantees that
$\deg(\phi)\in\Lambda_+\oplus\BN$ equals
$(\lambda'-w\lambda';(\alpha,\lambda'))$. Second, recall that
$p_2(\phi)\in\oZ^\alpha_{G_\zeta,J_\zeta}$ may be viewed as a
trivial $G_\zeta$-bundle on $\BP^1$ equipped with a $P_{J_\zeta}$-structure,
equal to $P_{J_\zeta}^{w_0}$ at $\infty\in\BP^1$. This $P_{J_\zeta}$-structure
is generically transversal to the (trivial) $^{J_\zeta}U_+$-structure,
equal to $^{J_\zeta}U_+$ at $\infty\in\BP^1$. These two structures define another
trivialization of the $G_\zeta$-bundle away from the points of nontransversality
$\unl{z}\subset\BA^1$, i.e.\ a point of $\on{Gr}_{G_\zeta,\unl{z}}$.
Conversely, the $P_{J_\zeta}$-structure extended from $\infty\in\BP^1$ to
$\BP^1\setminus\unl{z}$ via this
trivialization, and then to the whole of $\BP^1$ via the properness of
$^{J_\zeta}_{\hphantom{j}\zeta}\CB$ is nothing but the initial map
$\BP^1\to\ ^{J_\zeta}_{\hphantom{j}\zeta}\CB$. It follows that the relative position
of this $P_{J_\zeta}$-structure at $0\in\BP^1$, and the $B_+$-structure
(extended from $\infty\in\BP^1$ via the first (trivial) trivialization)
is nothing but the class $w\pmod{_\zeta W_{J_\zeta}}$ for
$\widehat{w}=-\alpha\cdot w$.
\end{proof}

\section{Twisted Uhlenbeck spaces}\label{twi}

\subsection{Root systems and foldings}\label{fold}
We recall the setup of~\cite[\S2.1]{brfi}.
Let $\fg^\svee$ be a simple Lie algebra
with the corresponding adjoint Lie group $G^\vee$. Let $T^\vee$ be
a Cartan torus of $G^\vee$. We choose a Borel subgroup $B^\vee\supset
T^\vee$. It defines the set of simple roots $\{\alpha_i,\ i\in I\}$.
Let $G\supset T$ be the Langlands dual groups.
We set $d_i=\frac{(\alpha_i,\alpha_i)}{2}$.

We realize $\fg^\svee$ as a \emph{folding} of a simple
simply laced Lie algebra
$\fg^{\svee\prime}$, i.e.\ as invariants of an outer automorphism $\sigma$ of
$\fg^{\svee\prime}$ preserving a Cartan subalgebra $\ft^{\svee\prime}\subset\fg^{\svee\prime}$
and acting on the root system of $(\fg^{\svee\prime},\ft^{\svee\prime})$.
In particular,
$\sigma$ gives rise to the same named automorphism of the Langlands dual
Lie algebras $\fg'\supset\ft'$.
\begin{NB}
  Langlands dual of $\fg^{\svee\prime}$ ? But $\fg^{\svee\prime}$ is
  simply-laced.
  \begin{NB2}
    Yes, you are right, but it seems more invariant to me this way.
    \end{NB2}%
\end{NB}%
We choose a $\sigma$-invariant Borel
subalgebra $\ft'\subset\fb'\subset\fg'$ such that $\fb=(\fb')^\sigma$.
The corresponding set of simple roots is denoted by $I'$.
We denote by $\Xi$ the finite cyclic group
generated by $\sigma$. Let $G'\supset T'$ denote the corresponding simply
connected Lie group and its Cartan torus.
The \emph{coinvariants} $X_*(T')_\sigma$ of $\sigma$ on the coroot
lattice $X_*(T')$ of $(\fg',\ft')$ coincide with the root lattice
of $\fg^\svee$. We have an injective map $a\colon X_*(T')_\sigma\to X_*(T')^\sigma$
from coinvariants to invariants defined as follows: given a coinvariant
$\alpha$ with a representative $\tilde\alpha\in X_*(T')$ we set
$a(\alpha):=\sum_{\xi\in\Xi}\xi(\tilde\alpha)$.

\subsection{Twisted Uhlenbeck} \label{twul}
We set $d=\on{ord}(\sigma)=\on{max}\{d_i\}$.
We fix a primitive root of unity $\zeta$ of order $d$.
We consider the affine plane $\BA^2=\BA^1\times\BA^1$ with coordinates $(z,t)$
(``horizontal'' and ``vertical''). We consider the moduli space $\on{Bun}_{G'}^b(\BA^1\times\BA^1)$ of
$G'$-bundles on $\BP^1\times\BP^1$, with second Chern class $b$, equipped with trivialization
at infinity $(\BP^1\times\BP^1)\setminus(\BA^1\times\BA^1)$. It has the Uhlenbeck closure
$\CU_{G'}^b\supset\on{Bun}_{G'}^b(\BA^1\times\BA^1)$~\cite{BFG}. The automorphism $\sigma$
acts on $\on{Bun}_{G'}^b(\BA^1\times\BA^1)$. Multiplication by $\zeta$
along the first (``horizontal'') copy of $\BA^1$ also acts there.

The \emph{twisted Uhlenbeck}
$\CU^a_\varsigma$ is defined as the closure in $\CU_{G'}^{da}$ of the fixed points of
the composition of the above automorphisms of
$\on{Bun}_{G'}^{da}(\BA^1\times\BA^1)$.

Alternatively, we set $\CK_t=\BC(\!(t^{-1})\!)\supset\CO_t=\BC[t]$, and
$\CK_z=\BC(\!(z^{-1})\!)\supset\CO_z=\BC[z]$.
The group ind-scheme
$G'(\CK_z)$ is equipped with an automorphism $\varsigma$ defined as the
composition of two automorphisms: a) $\sigma$ on $G'$; b) $z\mapsto\zeta z$.
We consider the twisted thick Grassmannian
$\BG\!\on{r}_{z,\varsigma}:=G'(\CK_z)^\varsigma/G'(\CO_z)^\varsigma$.
Then $\CU^a_\varsigma$ coincides with the Uhlenbeck closure of the space of degree $a$ based maps
from $\BP^1_t$ to $\BG\!\on{r}_{z,\varsigma}$, cf.~\cite{BFG}.

Finally, there is the third equivalent definition of $\CU^a_\varsigma$.
We consider the space of degree $da$ based \emph{twisted maps} from $\BP^1_z$ to
$\BG\!\on{r}_t:=G'(\CK_t)/G'(\CO_z)$: the fixed point set of the composition of two automorphisms
a) $z\mapsto\zeta z$ on the source $\BP^1_z$;
b) $\sigma$ on the target Grassmannian $\BG\!\on{r}_t$.
Its Uhlenbeck closure (cf.~\cite{BFG} and~\cite[\S2.4]{brfi}) coincides with $\CU^a_\varsigma$.

\subsection{Example}\label{exa}
It is known that $\CU^1_{G'}$ is isomorphic to $\BA^2\times\CN^{\on{min}}_{\fg'}$: the
product of the plane with the closure of the minimal nilpotent orbit.
The action of $\BC^\times\times\BC^\times$ is as follows:
$(q_1,q_2)\cdot(z,t,n)=(q_1z,q_2t,q_1q_2n)$. Hence the action of $\varsigma$
is as follows: $\varsigma(z,t,n)=(\zeta z,t,\zeta\sigma(n))$.
Say, let $G'=\on{SL}(2N),\ \fg^\svee={\mathfrak{sp}}(2N),\ \zeta=-1,\ d=2$.
Then the only fixed point of $\varsigma=-\sigma$ on $\CN^{\on{min}}_{\fg'}$ is 0
(as opposed to $(\CN^{\on{min}}_{\fg'})^\sigma=\CN^{\on{min}}_{\fg^\svee}$).
This motivates somewhat the degree $da$ in the definition of $\CU^a_\varsigma$.

\subsection{Degree twisted vs.\ untwisted}\label{twdeg}
We have a natural embedding $\iota\colon \BG\!\on{r}_\varsigma\hookrightarrow\BG\!\on{r}_{G'}$.
It induces
$\iota_*\colon H_2(\BG\!\on{r}_\varsigma,\BZ)\hookrightarrow H_2(\BG\!\on{r}_{G'},\BZ)$, and
$\iota^*\colon \on{Pic}(\BG\!\on{r}_{G'})\hookrightarrow\on{Pic}(\BG\!\on{r}_\varsigma)$.
All the four groups in question are $\BZ$ with canonical generators.
Now $\iota^*$ is an isomorphism, while $\iota_*$ is multiplication by $d$.

\subsection{Twisted zastava}\label{twzas}
Similarly to~\S\ref{twul} we define the twisted zastava $Z^\alpha_{\aff,\varsigma}$
as the closure in $Z^{a(\alpha)}_{\aff,G'}$ of the fixed points of $\varsigma$ in
$\oZ^{a(\alpha)}_{\aff,G'}$. Here $a(\alpha)$ is defined
in~\S\ref{fold}, though we rather need its obvious affine version. As before we can organize them into one ind-scheme
$Z_\varsigma$ endowed with an action of the torus $\TT_{\aff}$.

\subsection{Affine Lie algebra action}
Let $\grg_{\aff,\varsigma}$ denote the affine Lie algebra corresponding to $\grg',\sig$ and $\zeta$. In other words we do the following.
First, we consider the Lie algebra $\grg'[z,z^{-1}]$ and as before we extend $\sig$ to an automorphism $\varsigma$ of this algebra by combining the action of $\sig$
on $\grg'$ and the action on $z$ by multiplication by $\zeta$. Then we consider the algebra of invariants $\grg'[z,z^{-1}]^{\varsigma}$; the sought-for algebra $\grg_{\aff,\varsigma}$ is obtained from it by adding the loop rotation and then considering its central extension. This is an affine Lie algebra and
we let $\grg_{\aff,\varsigma}^\svee$ be its Langlands dual algebra.

Then analogously to point (1) of \S\ref{summary}
\begin{Conjecture}
Let $\calF_{\aff,\varsigma}=\oplus_\alpha \calF^\alpha_{\aff,\varsigma}$ where $\calF^{\alpha}_{\aff,\varsigma}$ is the IC-sheaf of $Z^\alpha_{\aff,\varsigma}$ (viewed as a perverse sheaf on $Z_{\aff,\varsigma}$).
Then we have a natural injective homomorphism
$U_\hbar(\grg_{\aff,\varsigma}^\svee)\to \Ext _{\TT_{\aff}}(\calF_{\aff,\varsigma},\calF_{\aff,\varsigma})$.
\end{Conjecture}

\subsection{Fixed points} Assuming the above conjecture, in order to compute geometrically the multiplicities of simple modules in Verma modules for the algebra $\grg_{\aff,\varsigma}^\svee$ we now need to analyze the fixed points with respect to various $\TT_{\aff,\lam}$ in the twisted zastava spaces $Z^{\alp}_{\aff,\varsigma}$. The fixed points (at positive, negative and critical levels) are described
as in~\S\ref{affine case} in terms of the twisted flag varieties $\CB_\varsigma,\unl{\CB}{}_\varsigma$
(which are nothing but certain Kac-Moody flag varieties). The relevant
affine Weyl groups sometimes coincide with some other affine Weyl groups
(``dual'') --  cf.~\cite{Lu2}, and here is a geometric explanation of this
coincidence.

Recall that $\lambda$ is a rational coweight in
\[(X_*(T)\oplus\BZ)\otimes\BQ=(X_*(T')_\sigma\oplus\BZ)\otimes\BQ.\] We write
it in the form $\lambda=(\bar\lambda,k)$. We write $\CB_\sigma$ (resp.\ $\unl{\CB}{}_\sigma$)
for the fixed points of $\sigma$ on the corresponding thin (resp.\ thick) affine flag varieties
of $G'$ (the untwisted flag varieties of the simply connected cover $G^\vee{}^{\on{sc}}$).

\begin{Lemma}
  \label{untw}
Suppose the denominator of $k$ is divisible by $d$ (in particular, $k\ne0$,
that is, the level is not critical). Then the fixed point set
$(Z^\alpha_{\on{aff},\varsigma})^{\TT_{\aff,\lambda}}$ is described as
in~Proposition~\ref{aff pos par},~Proposition~\ref{aff neg par} but in terms of the untwisted
flag varieties $\CB_\sigma,\unl{\CB}{}_\sigma$.
\end{Lemma}

\begin{proof}
Recall from the proof of~Lemma~\ref{phix} that the map from the fixed point set to
the flag variety is defined as $\phi\mapsto\phi(1)$. The flag variety in
question is $\CB_{G'}$ (resp.\ $\unl{\CB}{}_{G'})$.
Unraveling the definition of the
denominator of $k$ in~\S\ref{rat} we see that if we write $\phi$ as a function of $z$,
then it is actually a function of $z^d$. In particular, $\phi(\sqrt[d]{1})=\phi(1)$.
Hence this common value lies in the $\sigma$-fixed point set of $\CB_{G'}$
(resp.\ of $\unl{\CB}{}_{G'}$), that is in $\CB_\sigma$ (resp.\ in $\unl{\CB}{}_\sigma$).
\end{proof}


\bibliographystyle{myamsalpha}
\bibliography{kl}

\newcommand{\etalchar}[1]{$^{#1}$}
\providecommand{\noopsort}[1]{}\def\cftil#1{\ifmmode\setbox7\hbox{$\accent"5E#1$}\else
  \setbox7\hbox{\accent"5E#1}\penalty 10000\relax\fi\raise 1\ht7
  \hbox{\lower1.15ex\hbox to 1\wd7{\hss\accent"7E\hss}}\penalty 10000
  \hskip-1\wd7\penalty 10000\box7}
\providecommand{\bysame}{\leavevmode\hbox to3em{\hrulefill}\thinspace}
\providecommand{\MR}{\relax\ifhmode\unskip\space\fi MR }
\providecommand{\MRhref}[2]{%
  \href{http://www.ams.org/mathscinet-getitem?mr=#1}{#2}
}
\providecommand{\href}[2]{#2}
\begin{thebibliography}{FKMM99}

\bibitem[ABB{\etalchar{+}}05]{abbgm}
S.~Arkhipov, R.~Bezrukavnikov, A.~Braverman, D.~Gaitsgory, and I.~Mirkovi{\'c},
  \emph{Modules over the small quantum group and semi-infinite flag manifold},
  Transform. Groups \textbf{10} (2005), no.~3-4, 279--362.

\bibitem[AG03]{AG}
S.~Arkhipov and D.~Gaitsgory, \emph{Another realization of the category of
  modules over the small quantum group}, Adv. Math. \textbf{173} (2003),
  114--143.

\bibitem[AG15]{ArGa}
D.~Arinkin and D.~Gaitsgory, \emph{Singular support of coherent sheaves and the
  geometric {L}anglands conjecture}, Selecta Math. (N.S.) \textbf{21} (2015),
  no.~1, 1--199.

\bibitem[Ara07]{arakawa}
T.~Arakawa, \emph{Representation theory of {$\mathcal W$}-algebras}, Invent.
  Math. \textbf{169} (2007), 219--320.

\bibitem[BB93]{jantzen}
A.~Beilinson and J.~Bernstein, \emph{A proof of {J}antzen conjectures}, Adv.
  Soviet Math. \textbf{16} (1993), 1--50.

\bibitem[BF08]{BF}
R.~Bezrukavnikov and M.~Finkelberg, \emph{Equivariant {S}atake category and
  {K}ostant-{W}hittaker reduction}, Moscow Math. Jour. \textbf{8} (2008),
  39--72.

\bibitem[BF10]{BFI}
A.~Braverman and M.~Finkelberg, \emph{Pursuing the double affine
  {G}rassmannian. {I}. {T}ransversal slices via instantons on
  {$A_k$}-singularities}, Duke Math. J. \textbf{152} (2010), no.~2, 175--206.

\bibitem[BF17]{brfi}
\bysame, \emph{{Twisted zastava and q-Whittaker functions}}, J. Lond. Math.
  Soc. (2) \textbf{96} (2017), no.~2, 309--325.

\bibitem[BFG06]{BFG}
A.~Braverman, M.~Finkelberg, and D.~Gaitsgory, \emph{Uhlenbeck spaces via
  affine {L}ie algebras}, The unity of mathematics, Progr. Math., vol. 244,
  Birkh\"auser Boston, Boston, MA, 2006, see
  \url{http://arxiv.org/abs/math/0301176} for erratum, pp.~17--135. \MR{2181803
  (2007f:14008)}

\bibitem[BFGM02]{BFGM}
A.~Braverman, M.~Finkelberg, D.~Gaitsgory, and I.~Mirkovi{\'c},
  \emph{Intersection cohomology of {D}rinfeld's compactifications}, Selecta
  Math. (N.S.) \textbf{8} (2002), no.~3, 381--418, see
  \url{http://arxiv.org/abs/math/0012129v3} or Selecta Math. (N.S.) {\bf 10}
  (2004), 429--430, for erratum.

\bibitem[BFN16]{2014arXiv1406.2381B}
A.~{Braverman}, M.~{Finkelberg}, and H.~{Nakajima}, \emph{{Instanton moduli
  spaces and $\mathscr W$-algebras}}, Ast\'erisque (2016), no.~385, vii+128.
  \MR{3592485}

\bibitem[BG02]{BG}
A.~Braverman and D.~Gaitsgory, \emph{Geometric {E}isenstein series}, Invent.
  Math. \textbf{150} (2002), 287--384.

\bibitem[BGS96]{BGS}
A.~Beilinson, V.~Ginzburg, and W.~Soergel, \emph{Koszul duality patterns in
  representation theory}, J. Amer. Math. Soc. \textbf{9} (1996), no.~2,
  473--527.

\bibitem[Bra03]{Braden}
T.~Braden, \emph{Hyperbolic localization of intersection cohomology},
  Transform. Groups \textbf{8} (2003), no.~3, 209--216.

\bibitem[Bra06]{Bra-icm}
A.~Braverman, \emph{Spaces of quasi-maps into the flag varieties and their
  applications}, International {C}ongress of {M}athematicians. {V}ol. {II},
  Eur. Math. Soc., Z\"urich, 2006, pp.~1145--1170. \MR{2275639 (2008i:14019)}

\bibitem[BV08]{BV}
A.~Beilinson and V.~Vologodsky, \emph{A {D}{G} guide to {V}oevodsky's motives},
  Geom. Funct. Anal. \textbf{17} (2008), no.~6, 1709--1787.

\bibitem[BY13]{BY}
R.~Bezrukavnikov and Z.~Yun, \emph{On {K}oszul duality for {K}ac-{M}oody
  groups}, Represent. Theory \textbf{17} (2013), 1--98.

\bibitem[CG97]{CG}
N.~Chriss and V.~Ginzburg, \emph{Representation theory and complex geometry},
  Birkh\"auser Boston Inc., Boston, MA, 1997. \MR{MR1433132 (98i:22021)}

\bibitem[DG14]{DrGa}
V.~Drinfeld and D.~Gaitsgory, \emph{On a theorem of {B}raden}, Transform.
  Groups \textbf{19} (2014), no.~2, 313--358.

\bibitem[FG09]{FrGa}
E.~Frenkel and D.~Gaitsgory, \emph{{$D$}-modules on the affine flag variety and
  representations of affine {K}ac-{M}oody algebras}, Represent. Theory
  \textbf{13} (2009), 470--608.

\bibitem[FKMM99]{fkmm}
M.~Finkelberg, A.~Kuznetsov, N.~Markarian, and I.~Mirkovi{\'c}, \emph{A note on
  the symplectic structure on the space of {$G$}-monopoles}, Comm. Math. Phys.
  \textbf{201} (1999), no.~2, 411--421, see
  \url{http://arxiv.org/abs/math/9803124v6} or Comm. Math. Phys. {\bf 334}
  (2015), no. 2, 1153--1155, for erratum.

\bibitem[{Gai}05]{Ga05}
D.~{Gaitsgory}, \emph{{The notion of category over an algebraic stack}}, ArXiv
  e-prints (2005), \href{http://arxiv.org/abs/math/0507192}{{\ttfamily
  arXiv:math/0507192 [math.AG]}}.

\bibitem[Gai15]{1-aff}
D.~Gaitsgory, \emph{Sheaves of categories and the notion of 1-affineness},
  Contemp. Math. \textbf{643} (2015), 127--225.

\bibitem[GKM98]{GKM}
M.~Goresky, R.~Kottwitz, and R.~MacPherson, \emph{Equivariant cohomology,
  {K}oszul duality, and the localization theorem}, Invent. Math. \textbf{131}
  (1998), no.~1, 25--83.

\bibitem[GV93]{GinzburgVasserot}
V.~Ginzburg and {\'E}.~Vasserot, \emph{Langlands reciprocity for affine quantum
  groups of type {$A_n$}}, Internat. Math. Res. Notices (1993), no.~3, 67--85.
  \MR{1208827 (94j:17011)}

\bibitem[KT96]{KTneg}
M.~Kashiwara and T.~Tanisaki, \emph{Kazhan-{L}usztig conjecture for affine
  {L}ie algebras with negative level. {II}. {N}onintegral case}, Duke Math. J.
  \textbf{84} (1996), no.~3, 771--813.

\bibitem[KT98]{KTpos}
\bysame, \emph{Kazhan-{L}usztig conjecture for symmetrizable {K}ac-{M}oody
  algebras. {III}. {P}ositive rational case}, Asian J. Math. \textbf{2} (1998),
  no.~4, 779--832.

\bibitem[Lus81]{Lu3}
G.~Lusztig, \emph{Hecke algebras and {J}antzen's generic decomposition
  patterns}, Adv. in Math. \textbf{37} (1981), no.~2, 121--164.

\bibitem[Lus84]{reductive}
\bysame, \emph{Characters of reductive groups over a finite field}, Annals of
  Mathematical Studies, 107, Princeton University Press, 1984.

\bibitem[Lus94]{Lu2}
\bysame, \emph{Monodromic systems on affine flag manifolds}, Proc. Roy. Soc.
  London Ser. A \textbf{445} (1994), no.~1923, 231--246, see Proc. Roy. Soc.
  London Ser. A {\bf 450} (1995), no.~1940, 731--732, for erratum.

\bibitem[Lus95]{Lu-cus2}
\bysame, \emph{Cuspidal local systems and graded {H}ecke algebras. {II}},
  Representations of groups ({B}anff, {AB}, 1994), CMS Conf. Proc., vol.~16,
  Amer. Math. Soc., Providence, RI, 1995, pp.~217--275. \MR{1357201
  (96m:22038)}

\bibitem[Nak01]{Na-qaff}
H.~Nakajima, \emph{Quiver varieties and finite-dimensional representations of
  quantum affine algebras}, J. Amer. Math. Soc. \textbf{14} (2001), no.~1,
  145--238 (electronic). \MR{MR1808477 (2002i:17023)}

\bibitem[Nak12]{handsaw}
\bysame, \emph{Handsaw quiver varieties and finite {$W$}-algebras}, Mosc. Math.
  J. \textbf{12} (2012), no.~3, 633--666, 669--670. \MR{3024827}

\bibitem[Nak13]{tensor2}
\bysame, \emph{Quiver varieties and tensor products, {I}{I}}, Symmetries,
  Integrable Systems and Representations, Springer Proceedings in Mathematics
  \& Statistics, vol.~40, 2013, pp.~403--428.

\bibitem[Soe98]{Soergel}
W.~Soergel, \emph{Character formulas for tilting modules over {K}ac-{M}oody
  algebras}, Represent. Theory \textbf{2} (1998), 432--448.

\bibitem[Vas98]{Vasserot}
E.~Vasserot, \emph{Affine quantum groups and equivariant {$K$}-theory},
  Transform. Groups \textbf{3} (1998), no.~3, 269--299. \MR{1640675
  (99j:19007)}

\end{thebibliography}
\end{document}